\documentclass[11pt]{amsart}

\usepackage[english]{babel}
\usepackage[utf8]{inputenc}
\usepackage{amsmath}
\usepackage{amssymb}
\usepackage{amsfonts}
\usepackage{amsthm}
\usepackage{thmtools, thm-restate}
\usepackage{mathrsfs}
\usepackage[all]{xy}
\usepackage[pdftex]{graphicx}
\usepackage{color}
\usepackage{cite}
\usepackage{url}
\usepackage{indent first}
\usepackage[labelfont=bf,labelsep=period,justification=raggedright]{caption}
\usepackage[english]{babel}
\usepackage[utf8]{inputenc}
\usepackage[pdfencoding=auto, psdextra]{hyperref}
\usepackage[colorinlistoftodos]{todonotes}
\usepackage{caption}
\usepackage{subcaption}
\usepackage[capitalize,nameinlink]{cleveref}
\usepackage{centernot}
\usepackage{xcolor}
\hypersetup{
  colorlinks   = true, 
  urlcolor     = pink, 
  linkcolor    = blue, 
  citecolor   = green 
}
\usepackage[margin=1in]{geometry} 
\usepackage{tikz}
\usetikzlibrary{calc}

\DeclareMathOperator{\Aut}{Aut}

\DeclareMathOperator{\Bin}{Bin}

\DeclareMathOperator{\op}{op}

\DeclareMathOperator{\Rad}{Rad}

\DeclareMathOperator{\Var}{Var}

\newcommand{\hind}{\mathbb{I}}
\newcommand{\MP}{F_{H,G}}
\newcommand{\MPt}{F_{\triangle,G}}
\newcommand{\MPm}{F_{H,G,m}}
\newcommand{\infl}{\mathrm{inf}}
\newcommand{\eps}{\varepsilon}

\newcommand{\EE}{\mathbb{E}}

\newcommand{\NN}{\mathbb{N}}
\newcommand{\PP}{\mathbb{P}}

\newcommand{\RR}{\mathbb{R}}

\newcommand{\mcC}{\mathcal{C}}
\newcommand{\mcG}{\mathcal{G}}

\newcommand{\mcI}{\mathcal{I}}

\newcommand{\mcN}{\mathscr{N}}

\newcommand{\mcT}{\mathcal{T}}
\newcommand{\mcZ}{\mathcal{Z}}

\newcommand{\mfN}{\mathcal{N}}

\newcommand{\ind}{\mathbf{1}}

\newcommand{\w}{\mathbf{w}}
\newcommand{\rb}{\mathbf{w}}
\newcommand{\s}{\mathbf{s}}
\newcommand{\bz}{\mathbf{z}}
\newcommand{\bs}{\mathbf{s}}

\newcommand{\gj}{\mathcal{G}_H}

\newcommand{\nt}{Z_{\triangle}}

\makeatletter
\newcommand\thankssymb[1]{\textsuperscript{\@fnsymbol{#1}}}
\makeatother

\newenvironment{block}%
{\list{}{\leftmargin=0.0in\rightmargin=0.0in}  \item[]  }%
{\endlist}

\theoremstyle{plain}
\newtheorem{thm}{Theorem}[section]
\newtheorem{lemma}[thm]{Lemma}
\newtheorem{cor}[thm]{Corollary}
\newtheorem{conj}[thm]{Conjecture}
\newtheorem{prop}[thm]{Proposition}

\newtheorem{qn}[thm]{Question}

\newtheorem{obs}[thm]{Observation}

\theoremstyle{definition}
\newtheorem{defn}[thm]{Definition}

\theoremstyle{remark}

\newtheorem*{rmk}{Remark}
\newtheorem*{ex}{Example}

\numberwithin{equation}{section}
\numberwithin{thm}{section}

\title{Characterizing the fourth-moment phenomenon of monochromatic subgraph counts via influences}
\author[Mani]{Nitya Mani}
\author[Mikulincer]{Dan Mikulincer}
\address{Mani \& Mikulincer: Department of Mathematics, MIT}
\email{\{nmani,danmiku\}@mit.edu}
\begin{document}
	\maketitle
	\begin{abstract}
	    We investigate the distribution of monochromatic subgraph counts in random vertex $2$-colorings of large graphs. We give sufficient conditions for the asymptotic normality of these counts and demonstrate their essential necessity (particularly for monochromatic triangles). Our approach refines the fourth-moment theorem to establish new, local \textit{influence-based} conditions for asymptotic normality; these findings more generally provide insight into fourth-moment phenomena for a broader class of Rademacher and Gaussian polynomials.
	\end{abstract}
	\section{Introduction}
The study of local properties of large-scale graphs plays a pivotal role in computer science, statistics, and probability. In some cases, such local properties allow us to infer information about the global structure of the associated graph \cite{linial1993local}. In other cases, restricting to neighborhoods can reduce the complexity of statistical or computational tasks over the entire graphs~\cite{bdm22}. In this work, we focus on the task of subgraph counting and study the distribution of monochromatic subgraphs for random vertex colorings. 
 
	Consider a sequence of simple undirected graphs $G_n= (V(G_n), E(G_n))$, with vertex sets $V(G_n) = \{1, 2, \ldots, v(G_n)\}$ and edge sets $E(G_n)$. Further, let $H = (V(H), E(H))$ be some connected reference graph. Our object of study is the \emph{monochromatic subgraph count} statistic $T(H, G_n)$; the number of monochromatic copies of $H$ in $G_n$ with a random vertex coloring. That is, for some $c \in \NN$, independently color each vertex $v \in V(G_n)$ uniformly at random, using $c$ colors. We say that a copy of $H$ in $G_n$ is \textit{monochromatic} if all of its vertices have the same color. The statistic $T(H, G_n)$ is a sum of simple indicator random variables, and, depending on the graph structure, these indicators may appear independent, leading to a central limit theorem. We aim to make this intuition precise and establish simple (largely necessary and sufficient) conditions on the underlying graph sequence $G_n$ that imply the asymptotic normality of the statistic $T(H, G_n)$. (Other settings, where the limiting distribution is typically not Gaussian, e.g. the Poisson regime or dense $G_n$ regime, have also been studied~\cite{BHA19,bmm20,BHJ92,BHA17,JAEP}).

	As a motivating example, consider the case where $H = K_2$, i.e. $H$ is a single edge. For a random coloring of $V(G_n)$, $T(K_2,G_n)$ counts the number of monochromatic edges in $G_n$. If we envision $G_n$ as describing say, a social network, then $T(K_2, G_n)$ counts the number of pairs of friends in $G_n$ that have the same color, thus generalizing the classical birthday paradox. Such generalizations were previously studied in a variety of works~\cite{BHJ92,AGK16,CF06, D05,DM89,KMPT10,GH12,xiao2015universal}.
	The asymptotic distribution of $T(K_2,G_n)$ exhibits various universality phenomena. In particular, using the method of moments, Bhattacharya, Diaconis, and Mukherjee~\cite{BHA17} concluded that $T(K_2,G_n)$ is asymptotically normal whenever $e(G_n)\to \infty$ and the fourth moment of a suitably normalized version of $T(K_2,G_n)$ converges to $3$ (the fourth moment of the normal distribution).

	The principle that a central limit theorem for non-linear functions of random fields can arise as a consequence of convergence of the corresponding sequence of fourth moments arises in a variety of problems. The celebrated \textit{fourth-moment phenomenon} was first observed for Wiener-It\^o stochastic integrals in~\cite{NP05}, and~\cite{NP09} later proved a quantitative version. In subsequent years, the fourth-moment approach has appeared as a governing principle in the field of normal approximations. In particular, it has proven widely influential in proving central limit theorems in the Wiener space\footnote{The book~\cite{NP12} provides a survey on the topic, and the website \url{https://sites.google.com/site/malliavinstein/home} holds an up-to-date list of related results.}.
	
	Returning to the problem of monochromatic subgraph counts, beyond edges, the fourth-moment phenomenon was observed for monochromatic triangles~\cite{BHA22} when 
 $c \geq 5$, and for general monochromatic subgraphs $H$ when $c\geq 30$~\cite{DHM22}. Thus, when $c \geq 30$, in order for $T(H,G_n)$ to be asymptotically normal, it is necessary and sufficient that $\EE\left[T(H,G_n)^4\right]$ converges to an appropriate quantity. This idea has continued to be generalized in~\cite{JAEC}.
	
	It is tempting to expect that sufficiency of the fourth-moment condition continues to hold with a small number of colors. However, in the aforementioned results~\cite{BHA22,DHM22}, some lower bound on the number of colors turns out to be \textit{necessary} for the validity of the fourth-moment phenomenon.  The necessity can be seen in~\cite[Section 4.2]{BHA22} which provides examples of families of graphs $(G_n)_{n \in \NN}$, where the fourth moment of a suitably normalized version of $T(H, G_n)$ (using $2$ colors) converges to $3$, but $T(H, G_n)$ can be readily seen to not be asymptotically normal. These examples raise the following natural question:
	
	\begin{block}
		\centering
		\emph{Suppose that the vertices of $G_n$ are colored with a small number of colors. Can one find a necessary and sufficient condition on $G_n$ to ensure the asymptotic normality of $T(H,G_n)$?}
	\end{block}
	It is precisely this question that we tackle, focusing on $c = 2$, the smallest non-trivial number of colors.
	At the technical level, we refine and leverage the machinery of the fourth-moment theorem and adapt it to polynomial functionals of Bernoulli variables. This refinement reveals a new type of easy-to-verify \textit{local} properties of $G_n$. Combining these properties with established \emph{global} moment considerations allows us to distill the above-mentioned \textit{failures} of the fourth-moment phenomenon for $T(H, G_n)$. Our main contributions are as follows.
	\begin{enumerate}
		\item For triangles ($H = \triangle$), we completely characterize the asymptotic normality of $T(\triangle, G_n)$. We establish a new \emph{influence} condition that together with the fourth-moment condition determines whether $T(\triangle, G_n)$ is approximately normal. Our new condition is easy to check, is combinatorial in nature, and is inspired by the notion of influences from Boolean analysis. Informally, it requires that a single edge cannot appear in too many different triangles. 
		\item In the general setting, when $H$ is not a triangle, we show that the influence condition remains sufficient and conjecture that it is also necessary. We provide some evidence for this conjecture by relaxing the influence condition.
		\item We also prove new quantitative convergence rates for the normalized version of $T(H,G_n)$. When the influences of the graph are small, our rates turn out to be faster than the known results in the literature. Our result is rather general and applies to other related problems, such as monochromatic hypergraph counts and edge colorings. 
		\item Finally, as a technical tool that may be of independent interest, we extend the fourth-moment phenomenon to a broader class of Rademacher and Gaussian polynomials, which do not necessarily belong to a single chaos.
	\end{enumerate}

	\subsection{Setup and Main Results}
	Given a subgraph $H$ and a sequence of graphs $\{G_n\}_{n\in\NN}$, we shall henceforth use $T(H,G_n)$ to denote the number of monochromatic copies of $H$ in $G_n$ when we $2$-color the vertices $v \in V(G_n)$, uniformly at random, with independent colors $X_v \sim \Rad(\frac12)$\footnote{Here $\Rad(\frac12)$ stands for the Rademacher distribution, i.e. the uniform distribution on $\{-1,1\}$.}.  More formally, we have:
	\begin{align}\label{d:thgn}
		T(H, G_n) := \frac{1}{|\Aut(H)|} \sum_{\bs \in V(G_n)_{V(H)}} \ind\{X_{=\bs}\} \underbrace{\prod_{(i, j) \in E(H)} a_{s_i s_j}(G_n)}_{=: a_{H, \bs}(G_n)},
	\end{align}
	where
	\begin{itemize}
		\setlength\itemsep{0.3em}
		\item[--] $V(G_n)_r := \{ \bs = (s_1, \ldots, s_{r}) \in V(G_n)^r : s_i \text{ distinct} \}$.
		\item[--] $\Aut(H)$ is the automorphism group of $H$, the group of vertex permutations $\sigma$ such that $(i, j) \in E(H)$ if and only if $(\sigma(i), \sigma(j)) \in E(H)$.
		\item[--] We let $A(G_n) = (a_{ij}(G_n))_{i, j \in V(G_n)}$ be the adjacency matrix of $G_n$, so that $a_{ij} = \ind\{(i, j) \in E(G_n)\}$
		\item[--] $\ind\{X_{=\bs}\}:=\ind\{X_{s_1} = \cdots = X_{s_{r}}\}$ is the indicator variable that all of the vertices of $\bs\in V(G_n)_r$ have the same color.
	\end{itemize}
	Throughout, we assume $H$ is simple and connected and that $v(H) = r\ge 2$. We denote the centered and rescaled version of $T(H,G_n)$ by
	\begin{align}\label{d:zhgn}
		Z(H, G_n) := \frac{T(H, G_n) - \EE[T(H, G_n)]}{\sigma_H(G_n)},\quad \sigma_H(G_n) := \sqrt{\Var[T(H, G_n)]}.
	\end{align}
	We also let
	\begin{align}\label{d:nhgn}
		N(H,G_n):=\frac1{|\Aut(H)|}\sum_{s\in V(G_n)_{r}}\prod_{(i,j)\in E(H)} a_{s_is_j}(G_n),
	\end{align}
	be the number of copies of $H$ in $G_n.$
	
	The following definition will be instrumental in characterizing the asymptotic distribution of $Z(H,G_n)$.
	\begin{defn}[Influential vertices and edges] \label{def:inf}
		Given an integer $k \in [r]$ and $\rb=(w_1,\ldots,w_k) \in V(G_n)_k$ (ordered) let 
		$D_{\rb}(H, G_n)$ be the number of copies of $H$ in $G_n$ that include vertices $w_1,\ldots,w_k$.
		More precisely,
		\begin{align}\label{d:dw}
			D_{\rb}(H, G_n):=\frac1{\Aut(H)}\sum_{\s\in V(G_n)_r  :  \bs \supset \rb}  a_{H,\s}(G_n).
		\end{align}
		We say that vertex $w \in V(G_n)$ is an \textit{$\eps$-influential vertex} if
		$D_w(H, G_n) \ge \eps \sigma_H(G_n)$. 
		We say that a pair of vertices $\rb = (w_1, w_2) \in V(G_n)_2$ is \textit{$\eps$-influential pair} if $D_{\rb}(H, G_n) \ge \eps \sigma_H(G_n)$. If $e = \{w_1,w_2\} \in E(G_n)$, we will sometimes also call $\rb$ an \textit{$\eps$-influential edge} and denote its influence as $D_{e}(H, G_n)$.
	\end{defn}
	When $\eps$ can be taken to a be constant in $n$, we typically omit the ``$\eps$'' and simply refer to collections of vertices as \textit{influential}. As we shall show in \cref{lem:countinginfluence},  \cref{def:inf} is intimately related to the notion of influence from Boolean analysis~\cite[Chapter 2.2]{odonnell2014analysis}. This notion also bears similarity to the \textit{local counts} used in~\cite{bdm22} and a similar notion studied in~\cite{nourdin2010invariance} (discussed further in~\cref{sec:related}).\\

    \paragraph{\bf Quantitative central limit theorems under a low-influence condition.}
	Our main idea is to compare $Z(H,G_n)$ (which turns out to be a polynomial in Rademacher variables) to a polynomial in Gaussian variables. In some sense, we replace the $2$-coloring of $G_n$ with a continuous Gaussian ``coloring''. This comparison, facilitated by the \emph{Invariance Principle} of~\cite{mossel2010noise}, allows us to invoke known results about central limit theorems for polynomials in Gaussian spaces. In particular, we shall rely on the fourth-moment theorem of Nualart and Peccati~\cite{NP05}.
	
	The first hurdle we encounter is that known results in Gaussian spaces, such as the ones in \cite{NP05}, mostly revolve around polynomials belonging to a Wiener chaos, i.e. Hermite polynomials. Since $Z(H,G_n)$ is not a Hermite polynomial in general, we begin by extending the fourth-moment theorem to certain families of multi-linear polynomials that do not necessarily belong to a single Wiener chaos. For polynomials belonging to this family, we generalize the quantitative version of the fourth moment theorem given in~\cite{salim2011gaussian} to prove~\cref{thm:fourthmoment}, giving a necessary condition on such polynomials to ensure that the fourth-moment theorem holds.
	
	Among other examples,~\cref{thm:fourthmoment} holds for multi-linear polynomials with positive coefficients. Since $Z(H,G_n)$ is such a polynomial, this allows us to obtain our main quantitative convergence theorem, \cref{thm:mainquant}, which bounds the $W_1$ and $d_2$ distances (defined in~\cref{eq:wassdef,eq:dtowass}) between $Z(H, G)$ and a standard Gaussian in terms of the fourth-moment discrepancy $\EE[Z(H, G_n)^4] - 3$ and the maximum influence of a vertex $\max_{w \in V(G_n)} D_w(H, G_n).$
	
	Previous work establishing a fourth-moment theorem $Z(H, G_n)$ (for $H = \triangle$ in \cite{BHA22} and for general $H$ in \cite{DHM22}) used martingale central limit theorems to derive convergence rates in the \textit{Kolmogorov distance}:
	$$d_{\mathrm{kol}}(Z(H,G),\mcZ) := \sup\limits_{t \in \RR}\left|\PP\left(Z(H,G) \leq t\right) - \PP\left(\mcZ \leq t\right)\right|,$$
 where $\mcZ \sim \mcN(0, 1)$ is a standard Gaussian.
	Since $d_{\mathrm{kol}}(Z(H,G),\mcZ)$ can be bounded in terms of $W_1$ or $d_2$, informally, when there are no influential vertices, we have the following consequence of~\cref{thm:mainquant}.
	\begin{thm}[Informal] \label{thm:informal}
		Let $\{G_n\}$ be a sequence of graphs with no influential vertices (asymptotically) and let $H$ be a fixed connected simple graph. Then, 
		$$d_{\mathrm{kol}}(Z(H,G_n),\mcZ) \lesssim \min\left((\EE\left[Z(H,G_n)^4\right] - 3)^{\frac{1}{12}}, \max\limits_{w\in V(G_n)}D_w(H,G_n)(\EE\left[Z(H,G_n)^4\right] - 3)^{\frac{1}{8}}\right).$$
		Moreover, if $H = \triangle$ is a triangle, we obtain
		$$d_{\mathrm{kol}}(Z(\triangle,G_n),\mcZ) \lesssim (\EE\left[Z(\triangle,G_n)^4\right] - 3)^{\frac{1}{8}}.$$
	\end{thm} 
 \begin{rmk}
     \cref{thm:informal} is based on a rather general convergence result for polynomials in Rademacher random variables. Thus the obtained bounds also apply, mutatis mutandis, to similar problems. For example, a similar bound also holds if the colors are not chosen uniformly, or for variants in which we consider an edge coloring of $G_n$, rather than a vertex coloring, or when allowing $G_n$ to be a hypergraph. In another neighboring problem, related to motif estimation \cite{bdm22}, one asks to count the number of monochromatic subgraphs of a \emph{given} color. Up to the change in the definition of $Z(H,G_n)$, the bound in \cref{thm:informal} applies to this version verbatim. We further elaborate on this problem in \cref{sec:related}.
 \end{rmk}
	In the setting of $2$-coloring graphs with no influential vertices, we are able to bypass the difficulties encountered by the previous martingale arguments. Our result improves on the rates obtained in~\cite{DHM22}, where the authors achieved a convergence rate of $\approx (\EE\left[Z(H,G)^4\right] - 3)^{\frac{1}{20}}$. To our knowledge, these are the best-known rates in this setting. Notably, \cref{thm:informal} offers a substantial improvement when $H = \triangle$, arguably one of the more fundamental cases. Furthermore, as will be evident from \cref{thm:mainquant}, our rates further improve when considering the $W_1$ and $d_2$ metrics.\\
 
    \paragraph{\bf Characterizing the asymptotic normality of monochromatic triangles counts with pairwise influences.}
	In the above result, we impose the sufficient condition that $\{G_n\}$ lacks influential vertices. There are numerous examples of interest satisfying this condition, such as vertex-transitive graphs, some sparse graphs, and certain types of random graphs. However, a moment of thought will reveal that this condition is not necessary. For example, take $G_n$ to be a star on $n$ vertices, and $H = K_2$ an edge. Then, as a sum of independent variables $Z(K_2,G_n) \xrightarrow{n\to \infty}\mathcal{N}(0,1)$, yet the middle vertex is clearly influential. More generally, if each $G_n$ has a single influential vertex, by the symmetry of the colors (and conditioning on the color of $v$, $X_v$), one may see that \cref{thm:informal} continues to hold for $Z(H, G_n)$. Keeping this in mind, one might expect that if all influential vertices are ``sufficiently far away from each other'' (and so do not interact in fourth-moment calculations), then the fourth-moment theorem should continue to hold. This suggests that the true obstruction of a fourth moment theorem might be having subsets of nearby vertices $\rb$ that are as a collection, influential.
 
	In our next result, we apply this argument to triangles, when $H = \triangle$. In this case, we strengthen~\cref{thm:mainquant} by showing that a \textit{necessary and sufficient} condition for the fourth moment theorem to hold is that $G_n$ lacks \textit{influential edges}.
	\begin{thm}\label{thm:inf-edge}
		Consider a sequence of graphs $\{G_n\}$ with $\mcI_n$ the corresponding sets of influential vertices, and let $\sigma_{\triangle} = \sqrt{\mathrm{Var}(T(\triangle, G_n))}.$ Suppose that $\sup_n |\mcI_n| < \infty$ and $\mcZ \sim \mcN(0, 1)$ is a standard Gaussian.
		\begin{enumerate}
			\item If for every $n$, there exists $e \in E(G_n)$ such that 
			$$\lim\sup_{n \to \infty} \frac{D_e(\triangle, G_n)}{\sigma_{\triangle}} > 0$$
			Then,
			$$\lim \sup_{n \to \infty} d(Z(\triangle, G_n), \mcZ) > 0.$$
			\item Conversely, let $M_n := \max\limits_{e \in E(G_n)} \frac{D_e(\triangle, G_n)}{\sigma_{\triangle}}.$ Then, for $\mcZ \sim \mcN(0, 1)$,
			$$d_{\mathrm{kol}}(Z(\triangle,G_n),\mcZ) \lesssim \left(M_n^\frac{1}{2}  + (\EE\left[Z(\triangle, G_n)^4\right] - 3) \right)^{\frac{1}{5}}.$$
		
		\end{enumerate}
	\end{thm}
	Putting \cref{thm:inf-edge} in words, as long as the number of influential vertices is bounded uniformly over $n$, for $Z(\triangle, G_n)$ to be asymptotically normal, it is \emph{necessary and sufficient} that $G_n$ contains no influential edges and that  $\EE\left[Z(\triangle, G_n)^4\right] \xrightarrow{n \to \infty} 3.$ Thus, \cref{thm:inf-edge} gives a complete characterization of sequences $\{G_n\}$ for which the monochromatic triangle count is approximately normal. We remark that \cref{thm:inf-edge} includes the additional condition of having a bounded number of influential vertices. We regard this as a technical condition that seems to be an artifact of the proof of necessity.\\

    \paragraph{\bf Towards a characterization of general monochromatic subgraph counts.}
	Regarding more general subgraph counts, we establish sufficiency of having no influential pairs. Since every vertex participating in an influential pair is necessarily an influential vertex, this condition is weaker than the condition imposed by~\cref{thm:informal}, thus yielding a stronger result.
    \begin{restatable}{thm}{edge_general_case}\label{thm:inf-edge-general}
		Consider a sequence of graphs $\{G_n\}$, a graph $H$, and let $\sigma_{H} = \sigma_H(G_n).$ If $M_n := \max\limits_{\w \in V(G_n)_2} \frac{D_{\w}(\triangle, G_n)}{\sigma_{H}},$ then, for $\mcZ \sim \mcN(0, 1)$, 
		$$d_{\mathrm{kol}}(Z(H,G_n),\mcZ) \lesssim \left(M_n^2  + (\EE\left[Z(H, G_n)^4\right] - 3) \right)^{\frac{1}{20}}.$$
	\end{restatable}
    We now discuss possible generalizations of the first part of \cref{thm:inf-edge}. Our proof of the necessity condition for triangles proceeds by representing $Z(\triangle, G_n)$ as a quadratic polynomial in Rademacher variables. Such a quadratic representation is absent for general $H$ and so there is no way to generalize the proof verbatim. However, we conjecture that some adaptation to polynomials of higher degree should exist and that the above phenomenon holds more generally:
	\begin{conj}\label{c:main}
		Given a fixed graph $H$ and graph sequence $\{G_n\}$, if for all $N$, there exists some $n > N$ such that $G_n$ has an influential pair of vertices $\rb = (v, w) \in V(G_n)_2$, then $Z(H, G_n) \centernot{\xrightarrow{\text{in law}}} \mcN(0, 1)$. 
	\end{conj}
    We believe that a careful case analysis of possible tail behaviors of polynomials in Rademacher variables is the key to proving \cref{c:main}. While such an analysis seems to be beyond the scope of this work, we provide some evidence for its validity. In \cref{prop:stronginf} we show that the conjecture holds with respect to a stronger notion of influence. In this stronger notion, we ask that a single pair of vertices be incident to a constant fraction of the copies of $H$ in $G_n$. 

    One possible approach to prove \cref{c:main} is to establish some discrepancy between lower-order moments of $Z(H,G_n)$ and the standard Gaussian. However, the counterexamples to the fourth-moment phenomenon in \cite{BHA22} are all graphs with a single influential pair, where $\EE\left[Z(\triangle, G_n)^4\right] = 3$, in line with~\cref{thm:inf-edge}. In~\cref{s:exbad} we show that trying to look at higher fixed moments is likely also doomed to fail. We provide an explicit construction of a graph containing an influential edge, such that $Z(\triangle, G_n)$ matches the first $6$ moments of the standard Gaussian but where $Z(\triangle, G_n)$ is not asymptotically Gaussian. The implication here is that a solution to \cref{c:main} should probably look at more refined statistics of $Z(H,G_n)$ which capture the actual behavior of the tails.

    There is another natural approach to try which leads to a very natural question. When $c \geq 30$, using the results from~\cite{bdm22}, it is not hard to show that the existence of an influential pair implies $\EE\left[Z(H,G_n)^4\right]-3 > \eps > 0$. Hence with many colors, an influential pair precludes a CLT. This begs the following question about the monotonicity of the CLT with respect to the number of colors.
    \begin{qn} \label{qn:monoto}
        Let $c',c \in \NN$ with $c' > c$, and suppose that when using $c$ colors, $Z(H,G_n) \xrightarrow{\text{in law}} \mcN(0,1).$ Is it still true that $Z(H,G_n) \xrightarrow{\text{in law}} \mcN(0,1),$ when the vertices of $G_n$ are colored using $c'$ colors instead?
    \end{qn}
    A positive answer to \cref{qn:monoto} coupled with the previous observation would provide a positive answer to \cref{c:main}. It seems plausible that \cref{qn:monoto} is true; the more colors used the lower in magnitude the correlations between the different copies of $H$ in $G_n$. Common wisdom about the CLT suggests that it should be easier to approximate $Z(H,G_n)$ by a Gaussian in this case.

	\subsection{Related Works} \label{sec:related}

        The idea of representing $Z(H,G_n)$ as a polynomial in Rademacher variables appeared before 
        in~\cite{BHA17}. In the paper, the authors observed that for monochromatic edges, the fourth-moment phenomenon for $Z(K_2, G_n)$ could be reduced to a fourth-moment phenomenon for a homogeneous quadratic polynomial in Rademacher variables. Note, however, that in light of \cref{thm:inf-edge-general}, there is a qualitative difference between counting monochromatic edges and any other shape. When $H = K_2$, and $e(G_n) \to \infty$, no edge is influential. 
		
		More generally, fourth-moment theorems for polynomials in Rademacher random variables were studied before in \cite{dobler2019fourth} (see also \cite{blei2004rademacher} for an earlier result). In the paper, the authors focus on \emph{homogeneous} polynomials and employ a discrete variant of the Malliavin calculus to derive comparable bounds to our \cref{thm:mainquant}. Notably, the maximal influence also plays an important role in their results. Using similar techniques, the paper \cite{zhen2019peccati} later extended the convergence result to a multivariate setting in the asymptotic regime. As our approach is different and relies on the Invariance Principle, we are able to address a somewhat more general setting and consider non-homogeneous polynomials. Not only that, by restricting our attention to polynomials induced by monochromatic subgraph counts we identify the pairwise influence condition from \cref{thm:inf-edge} and \cref{thm:inf-edge-general}, rather than the single variable influence, as the main quantity governing the validity of the central limit theorem.
  
        We finish by discussing the related problem of subgraph sampling and motif estimation. In this problem, given a large graph $G$ and a reference graph $H$, one wishes to estimate the $N(H,G)$. To reduce the complexity, a natural approach is to subsample the vertices of $G$ with some probability $p$ and obtain a new graph $\tilde G$. It is then a standard procedure to obtain an unbiased estimator for $N(H,G)$ from $N(H,\tilde{G})$. To make the connection with the present setting, when coloring $G$ with $c$ colors, we take $p = \frac{1}{c}$, and $N(H,\tilde{G})$ becomes the number of monochromatic copies of $H$ of a \emph{given} fixed color. Thus, the subgraph sampling statistic is an asymmetric version of the monochromatic subgraph count statistic. From a statistical point of view, for a sequence of graphs $\{G_n\}$, establishing the asymptotic normality of $N(H,\tilde{G_n})$ is important for applications such as hypothesis testing.

        The asymptotic normality of $N(H,\tilde{G_n})$ was extensively studied in \cite{bdm22}, where the authors proved a fourth-moment theorem when $c \geq 20$ (equivalently when $p \leq \frac{1}{20}$). Similar to the case of monochromatic subgraphs, this result leaves open the question of asymptotic normality for small $c$. When $c = 2$, as mentioned above, $N(H,\tilde{G_n})$ can also be represented as a polynomial in Rademacher variables, and so Theorem \ref{thm:mainquant} applies almost verbatim. We thus get a fourth-moment theorem when $G$ does not have influential vertices. In contrast to our problem, due to the asymmetric nature of $N(H,\tilde{G_n})$ we do not expect that the existence, or lack, of an influential pair, will play an important role.  As an example, consider $G_n$ a star on $n$ vertices, which has an influential vertex and no influential pairs, and $H = K_2$ an edge. It is readily seen that $N(H,\tilde{G_n})$ is not asymptotically normal. For $N(H,\tilde{G_n})$, it seems reasonable to expect that the fourth-moment phenomenon will be completely determined by the existence of influential vertices.  
        \subsection*{Acknowledgements}
        We are grateful to Elchanan Mossel for enlightening discussions concerning the Invariance Principle and for being involved in the early stages of this project. We also thank Sayan Das and Zoe Himwich for helpful discussions. NM was supported by a Hertz Graduate Fellowship and the NSF GRFP DMS \#2141064.

	\section{Preliminaries}
	\subsection*{Notation}
	Let us first fix some notation we shall use in the sequel. Throughout $\{G_n\}$ is a sequence of simple undirected graphs, with $G_n = (V(G_n), E(G_n))$ having $v(G_n)$ vertices and $e(G_n)$ edges. We assume $e(G_n)\to\infty$ and take $H$ to be a simple connected graph on $r$ vertices and view $r$ as fixed when $n \to \infty$.
	We denote by $V(G_n)_r$ the set of ordered $r$-tuples of distinct vertices from $G_n$. We take $T(H, G_n), Z(H, G_n),$ and $N(H, G_n)$ as defined  in~\cref{d:thgn},~\cref{d:zhgn}, and~\cref{d:nhgn} respectively.
	
	We use $a(n) \lesssim b(n)$ to denote that there is some constant $C > 0$  such that $a(n) \le C b(n)$ for all $n$ and define $a(n) \gtrsim b(n)$ analogously. 
    We allow the implicit constant $C$ to depend on the graph $H$ and its size $r$, but not on other parameters such as $n$, or the graph $G_n$. 
    For $k \in \NN$ we set $[k] = \{1, 2, \ldots, k\}$, and for a set $S$ we denote by $\binom{S}{m}$ the collection of unordered $m$-subsets of $S$,
    and by $\binom{S}{\le m}$ the collection of unordered subsets of $S$ of size at most $m$.
	
	\subsection{Normal Approximations for Polynomials}
	Our aim in this section is to introduce some background and useful results concerning the central limit theorem for polynomial functionals. While some of the mentioned results hold in greater generality, we shall only focus on multi-linear polynomials with either Gaussian or Rademacher variables, as required in our setting. For this, we shall use the standard index notation for multi-linear polynomials. 
	
	For $I \subset [n]$ an index set, and $x = (x_1,\dots,x_n) \in \RR^n$, we set $x^I = \prod\limits_{i\in I}x_i$. Thus, every multi-linear polynomial $F:\RR^n \to \RR$ of degree $d$ can be written as, $F(x) = \sum\limits_{|I| \leq d}\alpha_Ix^I$, where $\{\alpha_I\}$ are some set of coefficients. If $F$ is \textit{homogeneous} of degree $d$, then $\alpha_I = 0$ whenever $|I| < d$.\\
	
	The first result we present is the \textit{fourth-moment theorem} of Nualart and Peccati~\cite{NP05}. In our setting, the theorem deals with a sequence $\{F_k\}_{k\geq 0}$ of homogeneous multi-linear polynomials of degree $d$ and an $n$-dimensional standard Gaussian vector $\mfN$. According to the theorem, as $k \to \infty$, $F_k(\mfN)$ converges in law to a Gaussian if and only if $\mathbb{E}\left[F_k(\mfN)^2\right] \xrightarrow{k\to\infty} \sigma^2$, for some $\sigma > 0$, and $\mathbb{E}\left[F_k(\mfN)^4\right] \xrightarrow{k\to\infty} 3\sigma^4$. The quantity $3\sigma^4$ is precisely the $4^{\mathrm{th}}$ moment of a Gaussian with variance $\sigma^2$. In this regard, the upshot of the theorem is the following: a priori, a central limit theorem for $F_k(N)$ necessitates the convergence of all moments; however, as long as the variance stabilizes, it suffices for the $4^{\mathrm{th}}$ moment to converge.
	
	We shall require a quantitative and multivariate form of the fourth-moment theorem, as it appears in \cite{salim2011gaussian}. To state this quantitative version we first need to introduce the \textit{Wasserstein distance} (the reader is referred to \cite[Chapter 6]{villani2009optimal} for an in-depth exposition of Wasserstein distances). If $X$ and $Y$ are two random vectors in $\RR^m$, we define their Wasserstein distance as,
	\begin{equation} \label{eq:wassdef}
		W_1(X,Y):= \inf\limits_{(X,Y)}\EE\left[\|X-Y\|\right],
	\end{equation}
	where the infimum runs over all couplings of $X$ and $Y$, that is, random vectors in $\RR^{2m}$ whose marginals on the first (resp. last) $m$ coordinates have the same law as $X$ (resp. $Y$). According to the Kantrovich-Rubinstein duality, the Wasserstein distance can also be formulated as the divergence over all $1$-Lipschitz functions,
	\begin{equation} \label{eq:KRduality}
		W_1(X,Y)= \sup\limits_{\substack{f:\RR^m \to \RR\\
				f \text{ is } 1\text{-Lipschitz}}}\left|\EE\left[f(X)\right] - \EE\left[f(Y)\right] \right|.
	\end{equation}
	In light of this characterization, we shall also consider the weaker and smoother $d_2$ distance:
	$$d_2(X,Y) = \sup\limits_{\substack{f:\RR^m \to \RR\\
			\|f\|_{C_2}\leq 1}}\left|\EE\left[f(X)\right] - \EE\left[f(Y)\right] \right|.$$
	Above  $\|f\|_{C_2}:= \max\{\|\nabla f\|_{\infty}, \|\nabla^2f\|_{\infty}\}.$ In other words, $\|f\|_{C_2} \leq 1$ implies that $f$ and $\nabla f$ are $1$-Lipschitz. By \eqref{eq:KRduality} it is now clear that,
	\begin{equation} \label{eq:dtowass}
		d_2(X,Y) \leq W_1(X,Y).    
	\end{equation}
	As will be soon evident, when dealing with random variables having small variances, it is sometimes beneficial to use the $d_2$ distance over the Wasserstein distance. For monochromatic subgraph counts, the $d_2$ distance will help us deal with influential shapes. A key property of both distances is that convergence in $W_1$, and hence also in $d_2$, implies convergence in law.
	\begin{thm} \label{thm:multifourthmoment}\textnormal{\cite[Theorem 1.5]{salim2011gaussian}}. Let $m \geq 1$, and for $i \in [m]$, let $F_i:\RR^n\to \RR$ be a homogeneous multi-linear polynomial of degree $d_i > 0$, with $d_i \neq d_j$ for $i \neq j$. Further let $\mfN \sim \mcN(0,\mathrm{I}_m)$, and define the $m\times m$ diagonal matrix 
		$\Sigma:= \mathrm{diag}\left(\EE\left[F_1(\mfN)^2\right],\dots,\EE\left[F_m(\mfN)^2\right]\right)$. 
		Taking $\widetilde{F} := (F_1(\mfN),\dots,F_m(\mfN)) \in \RR^m$ and $\widetilde{\mfN} \sim \mcN(0,\Sigma)$, if we let
		\begin{align*}
			\Delta(F) := & \Bigg[\sum\limits_{i=1}^m \sqrt{2d_i4^{d_i}\left(\EE\left[F_i(\mfN)^4\right] - 3\EE\left[F_i(\mfN)^2\right]^2\right)}\\
			+ &\sum\limits_{i \neq j}\bigg( \sqrt{2\EE\left[F_j(\mfN)^2\right]}\left(\EE\left[F_i(\mfN)^4\right] - 3\EE\left[F_i(\mfN)^2\right]^2\right)^{\frac{1}{4}}\\
			&\ \ \ \ \ \ \ + d_j2^{d_j}\sqrt{2(d_i + d_j)!\left(\EE\left[F_i(\mfN)^4\right] - 3\EE\left[F_i(\mfN)^2\right]^2\right)}\bigg)\Bigg],
		\end{align*}
		then
		\begin{align*}
			W_1(\widetilde{F},\widetilde{N}) &\leq  \sqrt{\|\Sigma\|_{\op}} \cdot \|\Sigma^{-1}\|_{\op} \cdot \Delta(F),\\
			d_2(\widetilde{F},\widetilde{N}) &\leq \Delta(F).\\
		\end{align*}
	\end{thm}
	\begin{rmk}
		We remark that~\cref{thm:multifourthmoment} is more general and applies to Hermite polynomials, or iterated stochastic integrals, on the Wiener space. Since multi-linear and homogeneous polynomials are also Hermite polynomials we chose to state the theorem this way, with slightly worse but simpler constants, as best befits our needs. 
	\end{rmk}
	The fourth moment theorem (\cref{thm:multifourthmoment}) deals with a Gaussian approximation of polynomials with Gaussian variables. However, in our problem, $T(H,G)$ is a polynomial in Rademacher random variables. We are thus interested in studying the law of $F(X)$ when $F:\RR^n \to \RR$ is some multi-linear polynomial and $X \sim \mathrm{Rad}(\frac{1}{2})^{\otimes n}$ is an $n$-dimensional random vector with independent Rademacher entries. As it turns out, the Gaussian is universal for this problem, in the sense that if $F(\mfN)$ is approximately Gaussian, then the same is also true for $F(X)$, c.f.~\cite{nourdin2010invariance}. The converse is not true in general though, and one can find a polynomial $F$, such that $F(X)$ can be well approximated by a Gaussian, while $F(\mfN)$ cannot. It is precisely these kinds of polynomials which we shall need to handle. To distinguish between the different cases, our main tool will be the \textit{Invariance Principle} of Mossel, O’Donnell, and  Oleszkiewicz~\cite{mossel2010noise}. 
	
	Before stating the theorem, for a multi-linear polynomial $F:\RR^n \to \RR$, $F(x) = \sum\limits_{|I| \leq d} \alpha_I x^I$, we define the \emph{influence} of a subset $J \subseteq [n]$ of variables, as
	\begin{equation} \label{eq:dinf}
		\infl_J(F) = \sum\limits_{J \subseteq I} \alpha_I^2.
	\end{equation}
	If $J =\{i\}$ is a singleton, we will also write $\infl_J(F) =  \infl_i(F).$ The Invariance Principle then says that as long as there are no influential variables, $F(\mfN)$ and $F(X)$ behave roughly the same.
	\begin{thm} \label{thm:invariance}\textnormal{\cite[Corollary 11.68]{odonnell2014analysis}}
		Let $F:\RR^n \to \RR$ be a multi-linear polynomial of degree $d$ and let $\mfN \sim \mcN(0,\mathrm{I}_n)$ and $X \sim \mathrm{Rad}(\frac{1}{2})^{\otimes n}$. Then, if $\mathrm{Var}(F(\mfN)) = 1$,
		$$W_1(F(X),F(\mfN)) \lesssim 2^d\left(\max\limits_{i\in [n]}\infl_i(F)\right)^{\frac{1}{4}}.$$
	\end{thm}
	\subsection{Monochromatic Counts as Multi-linear Polynomials} \label{sec:gpoly}We shall now show how the monochromatic count $T(H,G)$ can be represented as a multi-linear polynomial.
	To see this, for a fixed \emph{connected} graph $H$ on vertex set $V(H)$, and $r:= v(H)$ we define the $H$-indicator polynomial,  $\hind_H:\RR^{r} \to \RR$. For the definition of $\hind_H$, we index the coordinates of $\RR^{r}$ by $V(H)$, and write
	$$\hind_H(x) := \prod\limits_{v\in V(H)}(x_v +1) + (-1)^{r}\prod\limits_{v\in V(H)}(x_v -1) = 2 \sum\limits_{i=0}^{\lfloor \frac{r}{2} \rfloor}\sum\limits_{\substack{U \subset V(H)\\|U| = 2i}}\prod\limits_{u \in U}x_u.$$
	The equality follows by noticing that all odd-degree monomials must vanish.
	
	We now record several facts concerning the polynomial $\hind_H$. First, it is a multi-linear polynomial of degree at most $r$, and its $m$-homogeneous component is given by 
	\begin{equation} \label{eq:homogenousindicator}
		\hind_{H,m}(x) = \begin{cases}
			2\sum\limits_{\substack{U \subset V(H)\\|U| = m}}\prod\limits_{u \in U}x_u & m \equiv 0 \pmod 2 \\
			0 & m \equiv 1 \pmod 2.
		\end{cases},
	\end{equation}
	Second, it is indeed an indicator in the following sense: when restricted to $2$-colorings of $H$, encoded by $x \in \{\pm 1\}^{r}$, we have $\hind_H(x) = 2^{r}$ if $x = \mathbf{1} = (1,\dots,1)$ or $x=-\mathbf{1}$, and $\hind_H(x) = 0$ otherwise. Finally, $\hind_H$ does not depend on the actual graph structure of $H$, only on the number of vertices of $H$, and is invariant to permutations of its coordinates. 
	We note here that if $X \sim \mathrm{Rad}(\frac{1}{2})^{\otimes r}$, then 
	\begin{equation} \label{eq:polyexp}
		\EE[\hind_H(X)] = \hind_{H,0} = 2.
	\end{equation}
	\begin{ex}
		Suppose that $H = \triangle$ is a triangle on the vertex set $\{u,v,w\}$. In that case 
		\begin{equation} \label{eq:triexample}
			\hind_\triangle(x_u,x_v,x_w) = 2(x_ux_v + x_ux_w + x_vx_w + 1).
		\end{equation}
		It is easy to see that $\hind_\triangle (\ind) = \hind_\triangle (-\ind) = 8$, and that $\hind_\triangle (-1,1,1) = 0$, with the same being true for all other $2$-colorings.
	\end{ex}
	With this definition of $\hind_H$, we can now rewrite \eqref{d:thgn} as,
	\begin{equation} \label{eq:ttopoly}
		T(H, G) := \frac{1}{2^{r}|\Aut(H)|} \sum_{\bs \in V(G)_{r}} \hind_H(x\vert_{\bs}) a_{H, \bs}(G),
	\end{equation}
	where $x\vert_{\bs} = (x_v)_{v \in \bs}$. Equivalently, if $\mathcal{H}(G)$ stands for the set of all (unlabeled) copies of $H$ inside $G$ and $x \in \{\pm 1\}^{v(G)}$ is a $2$-coloring of $V(G)$ we can also write $T(H, G)$ in a more standard polynomial form as
	$$\MP(x) = \frac{1}{2^{r}}\sum\limits_{\bs \in \mathcal{H}(G)} \hind_H(x\vert_{\bs}).$$
	Note that the choice of ordering in $\bs$ does not matter, since $\hind_H$ is invariant to permutations of its coordinates.
	In light of~\cref{thm:invariance}, our next goal is to understand the influences of $\MP(x)$. Since the entries of $x$ are indexed by $V(G)$, we can identify index sets of $\MP$ and subsets $\rb \subset V(G)$ and write in multi-linear form, 
	$$\MP(x) = \sum\limits_{\rb \in {\binom{V(G)}{\le r}}} \alpha_{\rb}x^\rb.$$ Let us now express $\alpha_{\rb}$ in terms of the graph structure of $G$. 
	Looking at \eqref{eq:homogenousindicator}, it is clear that when $|\rb|$ is even, $x^\rb$ appears in the sum precisely once for every copy of $H$ in $G$ that includes $\rb$. In other words, $\alpha_{\rb} = \frac{1}{2^{r-1}}D_{\rb}(H,G)$, as in \eqref{d:dw}. When $|\rb|$ is odd then $\alpha_{\rb} = 0$, and we obtain the expression,
	\begin{equation} \label{eq:countingmultirep}
		\MP(x) = \frac{1}{2^{r-1}}\sum\limits_{i=0}^{\lfloor r/2\rfloor}\sum\limits_{\rb \in \binom{V(G)}{2i}} D_{\rb}(H,G)x^\rb.
	\end{equation}
	We can now make the connection between the influence of a single variable, as defined by \eqref{eq:dinf}, and the influence of a vertex, as in \eqref{d:dw}.
	\begin{lemma} \label{lem:countinginfluence}
		With the above notation, for any $v \in V(G)$,
		$$\infl_v(\MP) \leq  \frac{1}{2^{r-1}}D^2_v(H,G).$$
	\end{lemma}
	\begin{proof}
		From the representation in \eqref{eq:countingmultirep} we get,
		$$\infl_v(\MP) = \frac{1}{4^{r-1}}\sum\limits_{\substack{v \in \rb\\|\rb| \equiv 0 \text{ mod } 2}} D^2_{\rb}(H,G).$$
		By inclusion, it is also clear that for any $\rb \in V(G)$ that contains $v$, $D_{\rb}(H,G)\leq D_{v}(H,G)$. Moreover, for any $m \leq r$,
		$$\sum\limits_{\substack{v \in \rb\\|\rb| = m}} D_{\rb}(H,G) = \binom{r-1}{m-1}D_v(H,G).$$ 
		The coefficient $\binom{r-1}{m-1}$ counts the number of ways $v$ can be extended to a subset of size $m$ inside a copy of $H$. Combining the above observations, 
		$$\infl_v(\MP) \leq \frac{D_v(H,G)}{4^{r-1}}\sum\limits_{i=1}^{r/2}\sum\limits_{\substack{v \in \rb\\|\rb| = 2i}} D_{\rb}(H,G) = \frac{D^2_v(H,G)}{4^{r-1}}\sum\limits_{i=1}^{r/2} \binom{r-1}{2i-1} \leq \frac{D^2_v(H,G)}{2^{r-1}}.$$
	\end{proof}
	\section{A Fourth Moment Theorem for Monochromatic Counts}
	In this section, we prove our main quantitative convergence results, which show that for graphs with no influential vertices, the fourth-moment condition suffices for a central limit theorem to hold for $Z(H,G_n)$. Let  $F:\RR^n \to \RR$, $\mfN \sim \mcN(0,\mathrm{I}_n)$, and write $\widetilde{F} = \frac{F - \EE\left[F[\mfN]\right]}{\sqrt{\mathrm{Var}(F(\mfN))}}$, for the normalized version of the function $F$.
	\begin{thm} \label{thm:mainquant}
		Let $H$ be a fixed connected $r$-vertex graph and $G$ another graph. For $Z(H,G)$ defined in~\eqref{d:zhgn}, $\mfN \sim \mcN(0,\mathrm{I}_{v(G)})$, $X \sim \mathrm{Rad}(\frac{1}{2})^{\otimes r}$, and $\mcZ \sim \mcN(0,1)$,
		set $$m_4(H, G) := \sqrt{\EE\left[\widetilde\MP(\mfN)^4\right] - 3} +\left(\EE\left[\widetilde\MP(\mfN)^4\right] - 3\right)^{\frac{1}{4}},$$ where $\MP$ is as in \eqref{eq:countingmultirep}.
		Then, 
		\begin{align*}
			W_1(Z(H,G), \mcZ) &\lesssim \max\limits_{v\in V(G)}\left[\left(\frac{D_v^2(H,G)}{{\mathrm{Var}(\MP(X))}}\right)^{\frac{1}{4}} + D_{v}(H,G)m_4(H, G)\right],\\
			d_2(Z(H,G),\mcZ) &\lesssim \max\limits_{v\in V(G)}\left[\left(\frac{D_v^2(H,G)}{{\mathrm{Var}(\MP(X))}}\right)^{\frac{1}{4}} + m_4(H, G)\right].
		\end{align*}
	\end{thm}
	Let us make some remarks concerning~\cref{thm:mainquant}. First, if a sequence of graphs $\{G_n\}_{n\geq 0}$ has no influential vertices, then, by definition, $$\max\limits_{v\in V(G_n)}\left(\frac{D_v^2(H,G_n)}{{\mathrm{Var}(\MP(X))}}\right)^{\frac{1}{4}} \xrightarrow{n \to \infty} 0.$$
	Recall now that convergence in $d_2$ implies convergence in distribution. Hence for $Z(H,G_n) \xrightarrow[\text{in law}]{n \to \infty} \mcZ$ is it sufficient (and also necessary) that $m_4(G_n,H)\xrightarrow{n \to \infty} 0.$ 
	The fourth-moment discrepancy $m_4$ is defined in terms of the Gaussian measure. However, when there are no influential vertices, by~\cref{thm:invariance}, one could replace the Gaussian variables with Rademacher ones, expressing the fourth moment of $Z(H,G)$ and arrive asymptotically at the same expression.

	The proof of~\cref{thm:mainquant} will be conducted in several steps. We shall first prove a central limit theorem for the polynomial $\MP$ with Gaussian variables and then use the Invariance Principle from~\cref{thm:invariance} to replace the Gaussian variables with Rademacher ones.
	
	\subsection{Normal Approximation for non-homogeneous polynomials}
	We begin by extending the fourth-moment theorem to non-homogeneous polynomials. For a multi-linear polynomial of degree $d$, $F: \RR^n \to \RR$, given by $F(x) = \sum_{I\in {\binom{[n]}{\le d}}}\alpha_I x^I$,  denote by $F_m$ its degree $m$ homogeneous component, whenever $m \leq d$. In other words, $F_m(x) = \sum_{I\in \binom{[n]}{m}}\alpha_I x^I$.
	We now derive a necessary condition on multi-linear polynomials to ensure that the fourth-moment theorem holds.
	
	\begin{prop} \label{thm:fourthmoment}
		Let $F:\RR^n \to \RR$ be a degree $d$ multi-linear polynomial, and for $m \leq d$, let $F_m$ be its degree $m$ homogeneous component. Further, let $\mfN \sim \mcN(0, \mathrm{I}_n)$, and suppose that the following conditions are met:
		\begin{itemize}
			\item $\EE\left[F(\mfN)\right] = 0$ and $\EE\left[F(\mfN)^2\right] = 1$. 
			\item For every distinct $m,m',m'',m''' \leq d$, we have the following inequalities 
			\begin{align} \label{eq:conds}
				\EE&\left[F_{m}(\mfN)F^3_{m'}(\mfN)\right] \geq 0, \nonumber\\
				\EE&\left[F_{m}(\mfN)F_{m'}(\mfN)F^2_{m''}(\mfN)\right]\geq 0,\nonumber\\
				\EE&\left[F_{m}(\mfN)F_{m'}(\mfN)F_{m''}(\mfN)F_{m'''}(\mfN)\right]\geq 0.
			\end{align}
		\end{itemize}
		Then, if $\mcZ\sim \mcN(0,1),$
		\begin{align*}
			W_1(F(\mfN), \mcZ) &\leq C_d\cdot \max\limits_{m \leq d}\EE\left[F^2_m(\mfN)\right]^{-1} \cdot \left(\sqrt{\EE\left[F(\mfN)^4\right] - 3} +\left(\EE\left[F(\mfN)^4\right] - 3\right)^{\frac{1}{4}}\right),\\
			d_2(F(\mfN), \mcZ) &\leq C_d\left(\sqrt{\EE\left[F(\mfN)^4\right] - 3} +\left(\EE\left[F(\mfN)^4\right] - 3\right)^{\frac{1}{4}}\right),
		\end{align*}
		where $C_d > 0$ depends only on $d$.
	\end{prop}
	\begin{proof} 
		The main idea of the proof is to use~\cref{thm:multifourthmoment}. To utilize the theorem, we shall need to bound the fourth moment of $F(\mfN)$ in terms of the fourth moment of its homogeneous components. Towards this, we first note that the assumption $\EE[F(\mfN)] = 0$ implies $F_0 = 0$, and so,
		\begin{align*}
			\EE\left[F(\mfN)^4\right] =  \EE\left[\left(\sum\limits_{m=1}^dF_m(\mfN)\right)^4\right] \geq  \sum\limits_{m=1}^d \EE\left[F_m(\mfN)^4\right] + 6\sum\limits_{m \neq m'}\EE\left[F_m(\mfN)^2 F_{m'}(\mfN)^2\right],
		\end{align*}
		where the inequality follows since all missing terms were assumed to be positive. Now, observe that as multi-linear and homogeneous polynomials, for every $m \neq m'$, $F_m$ and $F_{m'}$ are Hermite polynomials and belong to some Wiener chaos. By \cite[Theorem 1]{malicet2016squared}, their squares are positively correlated, $\EE\left[F_m(\mfN)^2 F_{m'}(\mfN)^2\right] \geq \EE\left[F_m(\mfN)^2\right]\EE\left[F_{m'}(\mfN)^2\right].$ Combining this inequality with the previous bound yields,
		\begin{align*}
			\EE\left[F(\mfN)^4\right] \geq  \sum\limits_{m=1}^d \EE\left[F_m(\mfN)^4\right] + 6\sum\limits_{m \neq m'}\EE\left[F_m(\mfN)^2\right]\EE\left[F_{m'}(\mfN)^2\right].
		\end{align*}
		On the other hand, it is readily seen that $\EE\left[F_m(\mfN)F_{m'}(\mfN)\right]=0$, whenever $m \neq m'$, and so $1 = \EE\left[F(\mfN)^2\right] = \sum\limits_{m=1}^d\EE\left[F^2_m(\mfN)\right].$ So,
		$$3 = 3\left(\sum\limits_{m=1}^d \EE\left[F_m(\mfN)^2\right]\right)^2=  3\sum\limits_{m=1}^d \EE\left[F_m(\mfN)^2\right]^2+ 6\sum\limits_{m \neq m'}\EE\left[F_m(\mfN)^2\right] \left[F_{m'}(\mfN)^2\right],$$
		and
		\begin{equation} \label{eq:fourthmomentbound}
			\sum\limits_{m=1}^d \left(\EE\left[F_m(\mfN)^4\right] - 3\EE\left[F_m(\mfN)^2\right]^2\right) \leq \EE\left[F(\mfN)^4\right] - 3.
		\end{equation}
		Now, in order to invoke~\cref{thm:multifourthmoment}, let $\widetilde{\mfN}$ be a $d$-dimensional Gaussian with independent coordinates such that $\EE\left[\widetilde \mfN\right] = 0$ and $\EE\left[\widetilde \mfN_m^2\right] = \EE\left[F_m(\mfN)^2\right]$. Set, $\widetilde{F} = \{F_m(\mfN)\}_{m=1}^d$, $m(F) := \max\limits_m\EE\left[F^2_m(\mfN)\right]^{-1}$ and observe  $\max\limits_{m \leq d}\EE\left[F^2_m(\mfN)\right]\leq \EE\left[F^2(\mfN)\right] =1$. Thus, according to~\cref{thm:multifourthmoment}, if $\Delta(F)$ is as in the theorem,
		\begin{align*}
			\Delta(F) \leq & \sum\limits_{m=1}^d \sqrt{2d4^d\left(\EE\left[F_m(\mfN)^4\right] - 3\EE\left[F_m(\mfN)^2\right]^2\right)}\\
			+ &\sum\limits_{m \neq m'} \sqrt{2\EE\left[F_{m'}(\mfN)^2\right]}\left(\EE\left[F_m(\mfN)^4\right] - 3\EE\left[F_m(\mfN)^2\right]^2\right)^{\frac{1}{4}} + d2^d\sqrt{2(2d)!\left(\EE\left[F_m(\mfN)^4\right] - 3\EE\left[F_m(\mfN)^2\right]^2\right)}\\
			\leq &\sqrt{2d^24^d\sum\limits_{m=1}^d\left(\EE\left[F_m(\mfN)^4\right] - 3\EE\left[F_m(\mfN)^2\right]^2\right)}
			+2d\left(\sum\limits_{i=m}^d \EE\left[F_m(\mfN)^4\right] - 3\EE\left[F_m(\mfN)^2\right]^2\right)^{\frac{1}{4}}\\
			\ \ \ \ \ \ \ \ \ &+ \sqrt{8^d(2d)!\sum\limits_{m=1}^d\left(\EE\left[F_m(\mfN)^4\right] - 3\EE\left[F_m(\mfN)^2\right]^2\right)}\\
			\leq &\sqrt{4\cdot8^d(2d)!\left(\EE\left[F(\mfN)^4\right] - 3\right)} + 2d\left(\EE\left[F(\mfN)^4\right] - 3\right)^{\frac{1}{4}}.
		\end{align*}
		The second inequality uses the elementary inequality $\sum_{i=1}^d \sqrt{x_i} \leq \sqrt{d\sum_{i=1}^dx_i}$, valid whenever $x_i \geq 0$.  The third inequality is \eqref{eq:fourthmomentbound}.
		
		To finish the proof, suppose that $(\widetilde{F}, \widetilde{\mfN})$ are coupled with the optimal $W_1$ coupling, and let $\bf{1}$ be the all ones vector. So, by~\cref{thm:multifourthmoment}, since $F = \sum\limits_{m=1}^dF_m$,
		\begin{align*}
			W_1(F(\mfN), \mcZ) &\leq \EE\left[\left|\langle \widetilde{F}, {\bf{1}}\rangle - \langle \widetilde{\mfN}, {\bf 1}\rangle \right|\right] = \EE\left[\left|\langle \widetilde{F} - \widetilde{\mfN}, {\bf{1}}\rangle \right|\right]\\
			&\leq \|{\bf 1}\|\cdot \EE\left[\|\widetilde{F} - \widetilde{\mfN}\|\right] = \sqrt{d} \cdot W_1(\widetilde F, \widetilde \mfN) \leq C_d \cdot m(F) \cdot \Delta(F).
		\end{align*}
		A similar argument works for $d_2$, by noting that if $\|h\|_{C_2}\leq 1$ for some $h: \RR \to \RR$, then $\widetilde h(x):= h(\langle {\bf{1}}, x))$, satisfies $\|\widetilde h\|_{C_2} \leq d.$ 
	\end{proof}
	We now demonstrate several examples of polynomials that satisfy the conditions in \eqref{eq:conds}.
	\begin{lemma} \label{lem:allpos}
		Let $F: \RR^n \to \RR$ be a multi-linear polynomial of degree $d$, given by $F(x) = \sum\limits_{I \in {\binom{[n]}{\le d}}}\alpha_I x^I$. Suppose that for every $I \in {\binom{[n]}{\le d}}$, $\alpha_I \geq 0$, (resp. $\alpha_I \leq 0$), then the inequalities in \eqref{eq:conds} hold.
	\end{lemma}
	\begin{proof}
		The result follows from the following simple observation:    
		Let $I_1, I_2, I_3, I_4 \subset [n]$ be four sets of indices, not necessarily distinct, and let $\mfN \sim \mcN(0,\mathrm{I}_n)$. Then,
		$$\EE\left[\mfN^{I_1}\mfN^{I_2}\mfN^{I_3}\mfN^{I_4}\right]\geq 0.$$
		In other words, any monomial with Gaussian variables has a non-negative expectation. It is now enough to note that each of the expectations in \eqref{eq:conds} can be written as a sum of monomials in Gaussian variables. These monomials have positive coefficients since all coefficients in $F$ have equal signs and the claim follows.
	\end{proof}
	\begin{lemma}
		Let $F: \RR^n \to \RR$ be a quadratic multi-linear polynomial, $F(x) = F_1(x) + F_2(x)$. Then, the inequalities in \eqref{eq:conds} hold.
	\end{lemma}
	\begin{proof}
		It is enough to show that $\EE\left[F_1(\mfN)F^3_2(\mfN)\right] = \EE\left[F^3_1(\mfN)F_2(\mfN)\right] = 0$. Indeed, this holds since both $F_1(\mfN)F^3_2(\mfN)$ and $F^3_1(\mfN)F_2(\mfN)$ have odd degrees and so are anti-symmetric.
	\end{proof}
	
	\subsection{Graph Polynomials}
	Let us now show that the fourth-moment theorem holds for monochromatic counts in graphs with no influential vertices.
	For graphs $G$ and $H$, recall the polynomials $\hind_H$ and $\MP$ defined in \cref{sec:gpoly}.
	To apply~\cref{thm:fourthmoment}, we first need to estimate the Gaussian variance of the $m$ homogeneous component $\MPm$, for even $m$.
	\begin{lemma} \label{lem:homogenvar}
		Let $G$ and $H$ be graphs with $v(H) = r$. Then, if $\mfN \sim \mcN(0,\mathrm{I}_{v(G)})$, for any $0<m \leq r$ even,
		\begin{align*} 
			\frac{1}{4^{r-1}}\binom{r}{m} N(H,G) \leq \EE\left[\MPm(\mfN)^2\right]\leq \frac{1}{4^{r-1}}\binom{r}{m}N(H,G)\max\limits_{v\in V(G)}D_{v}(H,G),
		\end{align*}
		where $N(H,G)$ is the number of copies of $H$ in $G$, as in \eqref{d:nhgn}. 
		Moreover, for triangles, when $H = \triangle$,
		$\EE\left[F_{\triangle,G,2}(\mfN)^2\right] = \mathrm{Var}(F_{\triangle, G}(\mfN))$. 
	\end{lemma}
	\begin{proof}
		The representation in \eqref{eq:countingmultirep} along with a simple calculation shows that for even $m$
		$$\EE\left[\MPm(\mfN)^2\right] =\frac{1}{4^{r-1}}\sum\limits_{\rb \in {\binom{V(G)}{m}}}D^2_{\rb}(H,G).$$
		Since every copy of $H$ contains $\binom{r}{m}$ subsets of size $m$, we also have,
		$$\sum\limits_{\substack{\rb \in \binom{V(G)}{m}}} D_{\rb}(H,G) = \binom{r}{m}N(H,G).$$
		For any $\rb$ and $v \in \rb$, $D_{\rb}(H,G)\leq D_{v}(H,G)$, which establishes
		\begin{align*}
			\EE\left[\MPm(\mfN)^2\right]\leq \frac{1}{4^{r-1}}\max\limits_{v\in V(G)}D_{v}(H,G)\sum\limits_{{\rb \in \binom{V(G)}{m}}} D_{\rb}(H,G) = \frac{1}{4^{r-1}}\binom{r}{m}N(H,G)\max\limits_{v\in V(G)}D_{v}(H,G).
		\end{align*}
		Similarly, since $D_{v}(H,G)$ is a non-negative integer,
		\begin{align*}
			\frac{1}{4^{r-1}}\binom{r}{m} N(H,G) = \frac{1}{4^{r-1}}\sum\limits_{{\rb \in \binom{V(G)}{m}}} D_{\rb}(H,G) \leq \EE\left[\MPm(\mfN)^2\right],
		\end{align*}
		which concludes the proof in the general case.
		
		The result for triangles follows by observing that $F_{\triangle, G}$ is quadratic, as in \eqref{eq:triexample}. Thus,
		$$F_{\triangle, G} = F_{\triangle, G,0} + F_{\triangle, G,2},$$ with $F_{\triangle, G,0} = \EE\left[F_{\triangle, G}(\mfN)\right].$
	\end{proof}
	
	We can now prove a central limit theorem-type result for $\MP(\mfN)$.
	\begin{lemma} \label{lem:polyCLT}
		Let $G$ and $H$ be graphs with $v(H) = r$. Then, for $\mcN \sim \mcN(0,\mathrm{I}_{|V(G)|})$, $\mcZ \sim \mcN(0,1)$, and the normalized polynomial $\widetilde \MP(\mfN):=\frac{\MP(\mfN) - \EE\left[\MP(\mfN)\right]}{\sqrt{\mathrm{Var}(\MP(\mfN))}}$,
		\begin{align*}
			W_1(\widetilde \MP (\mfN),\mcZ) &\lesssim \max\limits_{v\in V(G)}D_{v}(H,G)\left(\sqrt{\EE\left[\widetilde\MP(\mfN)^4\right] - 3} +\left(\EE\left[\widetilde\MP(\mfN)^4\right] - 3\right)^{\frac{1}{4}}\right),\\
			d_2(\widetilde \MP (\mfN),\mcZ) &\lesssim \left(\sqrt{\EE\left[\widetilde\MP(\mfN)^4\right] - 3} +\left(\EE\left[\widetilde\MP(\mfN)^4\right] - 3\right)^{\frac{1}{4}}\right).
		\end{align*}
		Moreover, for triangles, when $H = \triangle$,
		$$W_1(\widetilde{F_{\triangle, G}} (\mfN),\mcZ) \lesssim \sqrt{\EE\left[\widetilde{F_{\triangle, G}}(\mfN)^4\right] - 3} +\left(\EE\left[\widetilde{F_{\triangle, G}}(\mfN)^4\right] - 3\right)^{\frac{1}{4}}.$$
	\end{lemma}
	\begin{proof}
		As observed from \eqref{eq:countingmultirep}, $\MP$ is a multi-linear polynomial of degree at most $r$, and all of its coefficients are positive. Thus,~\cref{lem:allpos} and~\cref{thm:fourthmoment} imply,
		\begin{align*}
			W_1(\widetilde \MP (\mfN),\mcZ) &\lesssim \max\limits_{m \leq r}\EE\left[\widetilde \MPm (\mfN)^2\right]^{-1}\left(\sqrt{\EE\left[\widetilde \MP(\mfN)^4\right] - 3} +\left(\EE\left[\widetilde \MP (\mfN)^4\right] - 3\right)^{\frac{1}{4}}\right),\\
			d_2(\widetilde \MP (\mfN),\mcZ) &\lesssim \left(\sqrt{\EE\left[\widetilde \MP(\mfN)^4\right] - 3} +\left(\EE\left[\widetilde \MP (\mfN)^4\right] - 3\right)^{\frac{1}{4}}\right)
		\end{align*}
		For triangles, by \cref{lem:homogenvar}, $\EE\left[\widetilde{F_{\triangle,G,2}}(\mfN)\right] = 1$, which resolves this case. For general $H$, it will suffice to show that
		\begin{equation} \label{eq:inftovar}
			\max\limits_{m \leq r} \EE\left[\widetilde \MPm (\mfN)^2\right]^{-1} \leq 2^r\max\limits_{v\in V(G)}D_{v}(H,G).
		\end{equation}
		By~\cref{lem:homogenvar}, for any $0<m\leq r$ even,
		$\EE\left[\MPm(\mfN)^2\right] \geq \frac{1}{4^{r-1}}\binom{r}{m}N(H,G)$.
		Similarly, since $\mathrm{Var}(\MP(\mfN)) = \sum\limits_{m=1}^{r}\EE\left[\MPm(\mfN)^2\right],$
		$$\mathrm{Var}(\MP(\mfN))\leq  \frac{1}{4^{r-1}}\max\limits_{v\in V(G)}D_{v}(H,G)\sum\limits_{m' = 1}^r\binom{r}{m'}N(H,G) \leq \frac{1}{2^{r-2}}\max\limits_{v\in V(G)}D_{v}(H,G)N(H,G).$$ 
		So, 
		$$\EE\left[\widetilde \MPm (\mfN)^2\right]^{-1} = \frac{\mathrm{Var}(\MP(\mfN))}{\EE\left[\MPm(\mfN)^2\right]} \leq \max\limits_{v\in V(G)}D_{v}(H,G)\frac{4^{r-1}N(H,G)}{2^{r-2}\binom{r}{m}N(H,G)},$$
		which is \eqref{eq:inftovar}.
	\end{proof}
	\subsection{A Fourth-Moment Theorem for Monochromatic Subgraph Counts} 
	We can now combine~\cref{lem:polyCLT} with the Invariance Principle in~\cref{thm:invariance} to prove a Fourth Moment Theorem for $2$-colorings.
	\begin{proof}[Proof of~\cref{thm:mainquant}]
		Let $X \sim \mathrm{Rad}(\frac{1}{2})^{\otimes v(G)}$. Since (c.f.~\eqref{eq:ttopoly}) $T(H, G) = F_{H, G}(X)$, we have
		$$Z(H,G) = \widetilde \MP(X) = \frac{\MP(X) - \EE\left[\MP(X)\right]}{\sqrt{\mathrm{Var}(\MP(X))}}.$$
		Thus, we have the following 
		\begin{align*}
			W_1(Z(H,G), \mcZ) &\leq W_1(\widetilde\MP(X),\widetilde\MP(\mfN)) + W_1(\widetilde\MP(\mfN),\mcZ) \tag{triangle ineq.}\\
			W_1(\widetilde\MP(X),\widetilde\MP(\mfN)) &\lesssim \left(\max\limits_{v\in V(G)}\infl_v(\widetilde \MP)\right)^{\frac{1}{4}} \tag{\cref{thm:invariance}}\\
			W_1(\widetilde\MP(\mfN),\mcZ)  &\lesssim \max\limits_{v\in V(G)}D_{v}(H,G)\left(\sqrt{\EE\left[\widetilde\MP(\mfN)^4\right] - 3} + \left(\EE\left[\widetilde\MP(\mfN)^4\right] - 3\right)^{\frac{1}{4}}\right) \tag{\cref{lem:polyCLT}}.
		\end{align*}
		We now observe that by~\cref{lem:countinginfluence},
		$$\max\limits_{v\in V(G)}\infl_v(\widetilde \MP) = \frac{\max\limits_{v\in V(G)}\infl_v(\MP)}{\mathrm{Var}(\MP(X))} \lesssim \frac{\max\limits_{v\in V(G)}D_v^2(H,G)}{{\mathrm{Var}(\MP(X))}}.$$
		Combining these bounds gives the desired result
		$$W_1(Z(H,G), \mcZ) \lesssim \max_{v \in V(G)} \left[\left( \frac{D_v^2(H,G)}{{\mathrm{Var}(\MP(X))}}\right)^{\frac14} + D_v(H, G)\left(\sqrt{\EE\left[\widetilde\MP(\mfN)^4\right] - 3} + \left(\EE\left[\widetilde\MP(\mfN)^4\right] - 3\right)^{\frac{1}{4}}\right)  \right].$$
		The proof for $d_2$ is essentially the same, by applying the $d_2$ bound in~\cref{lem:polyCLT}. We only need to note that by \eqref{eq:dtowass},
		$$d_2(\widetilde\MP(X),\widetilde\MP(\mfN)) \leq W_1(\widetilde\MP(X),\widetilde\MP(\mfN)).$$
	\end{proof}
	We can also use \cref{thm:mainquant} to obtain bounds on the Kolmogorov distance. By \cite[Theorem 3.3]{chen2011normal}, $d_{\mathrm{kol}}(Z(H,G), \mcZ) \leq \sqrt{2W_1(Z(H,G), \mcZ)}$ and in \cref{lem:koltod2} of the Appendix we prove $d_{\mathrm{kol}}(Z(H,G), \mcZ) \lesssim d_2(Z(H,G), \mcZ)^{\frac{1}{3}}$. Thus \cref{thm:mainquant} and these bounds immediately imply the following quantitative version of \cref{thm:informal}.
	\begin{cor} \label{cor:kolgeneral}
		In the same setting of \cref{thm:mainquant},
		\begin{align*}
			d_{\mathrm{kol}}&(Z(H,G), \mcZ) \lesssim \\
			&\max\limits_{v\in V(G)}\min\left(\left(\left(\frac{D_v^2(H,G)}{{\mathrm{Var}(\MP(X))}}\right)^{\frac{1}{4}} + D_{v}(H,G)m_4(H, G)\right)^{\frac{1}{2}},\left( \left(\frac{D_v^2(H,G)}{{\mathrm{Var}(\MP(X))}}\right)^{\frac{1}{4}} + m_4(H, G)\right)^{\frac{1}{3}}\right). 
		\end{align*}
	\end{cor}
	Similarly, by applying the Wasserstein bound to the triangle bound of \cref{lem:polyCLT} we obtain an improved convergence rate for triangles.
	\begin{cor} \label{cor:koltriangles}
		In the same setting of \cref{thm:mainquant}, if $H = \triangle$ is a triangle,
		\begin{align*}
			d_{\mathrm{kol}}&(Z(\triangle ,G), \mcZ) \lesssim \max\limits_{v\in V(G)}\left(\left(\frac{D_v^2(H,G)}{{\mathrm{Var}(\MP(X))}}\right)^{\frac{1}{4}} + m_4(H, G)\right)^{\frac{1}{2}}. 
		\end{align*}
	\end{cor}
	\section{A Closer Look at Monochromatic Triangles} \label{sec:triangles}
	In this section, we focus on monochromatic triangle counts, taking $H = \triangle$, and prove \cref{thm:inf-edge}. We recall the polynomial expression in~\eqref{eq:triexample}:
	$$\hind_\triangle(x_u,x_v,x_w) = 2(x_vx_u + x_vx_w + x_ux_w + 1).$$
	For graph $G$, let $\mcT(G)$ be the set of unlabeled triangles in $G$. Given a two-coloring $x\in \{\pm 1\}^{v(G)}$ define the centered, normalized count 
	$$\nt(x) := \widetilde\MPt(x) = \frac{1}{4\sigma_\triangle}\sum\limits_{T \in \mathcal{T}(G)} (\hind_\triangle(x\vert_T) - 2) = \frac{1}{\sigma_\triangle}\sum\limits_{uv \in E(G)} D_{uv}(\triangle, G)x_ux_v,$$
	where for $T = \{u,v,w\}$, $x\vert_T = \{x_v,x_u,x_w\}$, and
	\begin{equation} \label{eq:varrep}
		\sigma^2_\triangle := \mathrm{Var}(T(\triangle, G)) = \frac{1}{16}\sum\limits_{e\in E(G)}D^2_e(\triangle, G).
	\end{equation}
	The expression on the right-hand side follows from~\eqref{eq:countingmultirep}, noting that the only pairs of vertices $u, v \in V(G)$ where $D_{uv}(\triangle, G) > 0$ are those that form an edge $uv \in E(G).$
	
	We study the distribution of $\nt(X)$ when  $X \sim \mathrm{Rad}(\frac{1}{2})^{\otimes v(G)}$, so that $\nt(X) = Z(\triangle, G)$. It will be convenient to start by replacing $X$ with a Gaussian vector, and in light of~\cref{thm:invariance}, we begin by handling influential variables. If $\{G_n\}$ is a sequence of graphs let us denote by $\mathcal{I}_n$ the set of influential variables. Formally, $\mathcal{I}_n =(v_i^{(n)})_{i=1}^{k}\subset V(G_n)$ is an ordered tuple of vertices, such for any $i = [k]$ 
	$$\liminf\limits_{n \to \infty} \frac{D_{v_i^{(n)}}(\triangle, G_n)}{\sigma_\triangle} > 0,$$
	and for any other vertices $w^{(n)} \in V(G_n)$, where there exists $N$ such that for all $n > N,$ $w^{(n)} \not \in \mcI_n,$ we have 
	$$\limsup\limits_{n \to \infty} \frac{D_{w^{(n)}}(\triangle, G_n)}{\sigma_\triangle} = 0.$$
	Note that for simplicity, we are implicitly assuming that the number of influential vertices is uniformly bounded in $n$.
	\subsection{Influential Edges Forbid CLTs}
	Our first goal will be to show that when $G_n$ contains an influential edge, then the CLT cannot hold, and thus having no influential edge is a necessary condition for the asymptotic normality of $Z(\triangle, G_n)= Z_{\triangle}(X)$ when $X \sim \Rad(\frac12)^{\otimes v(G_n)}$. This is precisely the first item of \cref{thm:inf-edge}.
	\begin{prop} \label{prop:edgenoCLT}
		Let $\{G_n\}_{n\geq 0}$ be a sequence of graphs with $\mathcal{I}_n$ the sets of their influential vertices. Suppose that $\sup\limits_n|\mathcal{I}_n| = m < \infty$ and that there exists a sequence of influential edge $e_n\in E(G_n)$, such that
		$$\lim \sup\limits_{n\to \infty} \frac{D_{e_n}(\triangle,G_n)}{\sigma_\triangle} > 0.$$
		Then, if $\mcZ \sim \mcN(0, 1)$,
		$$\limsup\limits_{n\to \infty} d_{\mathrm{kol}}(Z(\triangle, G_n),\mcZ) > 0.$$
	\end{prop}
	
	Our strategy for dealing with influential variables in $\nt(x)$ will be to fix their values. For a partial coloring $\rho: {\mathcal{I}_n} \to \{\pm 1\}$ we write $\nt^{(\rho)}:V(G_n) \setminus \mathcal{I}_n \to \RR$ for the function $\nt$ where every $v \in \mathcal{I}_n$ is assigned the color $\rho(v)$.
	We can express $\nt^{(\rho)}$ in the following way
	\begin{align*}
		\nt^{(\rho)}(x) &= \frac{1}{3\sigma_\triangle}\left[\sum\limits_{u,v\in \mathcal{I}_n}D_{uv}(\triangle,G_n)\rho(u)\rho(v) + \sum\limits_{u\in \mathcal{I}_n, v\notin \mathcal{I}_n}D_{uv}(\triangle,G_n)\rho(u)x_v + \sum\limits_{u,v\notin \mathcal{I}_n}D_{uv}(\triangle,G)x_ux_v\right]\\
		&= \nt^{(0, \rho)} + \nt^{(1,\rho)}(x) + \nt^{(2)}(x).
	\end{align*}
	For $i = 0,1,2$, $\nt^{(i, \rho)}$ stands for the $i^{\mathrm{th}}$ homogeneous part of $\nt^{(\rho)}$. It is also clear from the expression that $\nt^{(2)}$ does not depend on $\rho$. Let us collect some facts about this decomposition when the variables $x_v$ is taken to be Gaussian.
	\begin{lemma} \label{lem:Hdecomp}
		Let $\mfN \sim \mcN(0, \mathrm{I}_{|V(G_n) \setminus \mathcal{I}_n|})$, then the following conditions hold:
		\begin{enumerate}
			\item There exists a sequence $\{\lambda_{i,n}\}_{i=1}^\infty \in \RR$ with $|\lambda_{1, n}| \ge |\lambda_{2, n}| \ge \cdots $ and independent standard Gaussians $\{M_i\}_{i=1}^\infty$ such that,
			$$\nt^{(2)}(\mfN) \stackrel{\mathrm{law}}{=} \sum\limits_{i=1}^\infty \lambda_{i,n}(M_i^2 - 1).$$
			\item For every partial coloring $\rho$, there exists $\sigma_\rho^2 \in [0,1]$, such that $\nt^{(1,\rho)}(\mfN) \sim \mcN(0,\sigma^2_\rho).$  
			\item $\nt^{(0, \rho)} \xrightarrow{n\to \infty} 0$ for every $\rho$ if and only if there does not exist an influential edge.
		\end{enumerate}
		\begin{proof}
			For the first assertion, we note that $\nt^{(2)}(\mfN)$ is a quadratic form in Gaussian variables. The result is obtained by diagonalizing the form and using the fact that the Gaussian is rotationally invariant, see \cite[Proposition 2.7.13]{NP12} for more details.
			
			The second assertion follows since $\nt^{(1, \rho)}(\mfN)$ is a centered linear functional of Gaussian variables and hence Gaussian. The fact that $\sigma_\rho^2 \leq 1$ follows by calculating the variance and comparing to \eqref{eq:varrep}.
			
			For the last item, let $R \sim \Rad(\frac12)^{\otimes  
				|\mathcal{I}_n|}$ be a random coloring of the influential vertices. Since $R$ is centered, the condition $\nt^{(0, R)} \xrightarrow{n\to \infty} 0$ is equivalent to $\mathrm{Var}_R(\nt^{(0, R)}) \xrightarrow{n\to \infty} 0.$ Since $R(u)R(u')$ is uncorrelated with $R(v)R(v')$ whenever $\{u,u'\} \neq \{v,v'\}$ we get that 
			$$\mathrm{Var}_R(\nt^{(0, R)}) = \frac{1}{16\sigma^2_\triangle}\sum\limits_{uv\in \mathcal{I}_n}D^2_{uv}(\triangle,G_n).$$
			Suppose that we have an influential edge i.e. $u_n,v_n \in \mathcal{I}_n$ such that
			$$\liminf\limits_{n\to \infty}\frac{D_{u_nv_n}(\triangle,G_n)}{\sigma_\triangle} > 0.$$
			In this case it is clear that $\liminf\limits_{n\to \infty}\mathrm{Var}_R(\nt^{(0, R)}) >0$ as well. On the other hand, if there are no influential edges then,
			\begin{align*}
				\mathrm{Var}_R(\nt^{(0, R)}) &\leq \left(\max\limits_{u,v \in V(G)} \frac{D_{uv}(\triangle,G_n)}{\sigma_\triangle}\right)\frac{1}{16\sigma_\triangle}\sum\limits_{u,v\in \mathcal{I}_n}D_{uv}(\triangle,G_n) = o(1),
			\end{align*}
			where we have used the variance representation \eqref{eq:varrep}, and that $\sup\limits_n|\mathcal{I}_n| = m <\infty$.
		\end{proof}
	\end{lemma}
	
	We can now establish one particular case where $\nt(\mfN)$ does not converge to a Gaussian. If there is at least one large coefficient in the quadratic decomposition of $\nt^{(2)}$, then $\nt(X)$ will have a sub-exponential behavior and so cannot converge to a Gaussian.
	\begin{lemma} \label{lem:Hsubexp}
		Let $\{G_n\}_{n\geq 0}$ be a sequence of graphs with $\{\mathcal{I}_n\}$ the sets of their influential vertices. Suppose that $\sup\limits_n|\mathcal{I}_n| = m < \infty$. Further let $\lambda_{1,n}$ be as in~\cref{lem:Hdecomp}, and assume
		$$\lim\limits_{n\to \infty}{|\lambda_{1,n}|} = c> 0.$$
		Then, there exists some constant $c'>0$, such that for every partial coloring $\rho :\mathcal{I}_n \to \{\pm 1\}$, and all $t>0$ large enough,
		$$\PP\left(|\nt^{(\rho)}(\mfN)| > t\right) \geq e^{-c't},$$ 
		where $\mfN \sim \mcN(0,\mathrm{I}_{|V(G_n) \setminus \mathcal{I}_n|})$.
	\end{lemma}
	\begin{proof}
		We fix some large $t > 10$, and bound $\PP\left(|\nt^{(2)}(\mfN)| \geq t\right)$ from below. Towards that, recall the expression $\nt^{(2)}(\mfN)$ from~\cref{lem:Hdecomp}, since we may assume $\lambda_{1,n}\geq \frac{c}{2}$, we obtain
		\begin{align*}
			\PP\left(|\nt^{(2)}(\mfN)| \geq t\right) &\geq  \PP\left(\lambda_{1,n}(M_1^2 - 1) \geq t+4\right)\PP\left(|\nt^{(2)}(\mfN) - \lambda_{1,n}(M_1^2 - 1)| \leq 4\right)\\
			&
			\geq \frac{1}{2}\PP\left(M_1^2 \geq \frac{t+4+\lambda_{1,n}}{\lambda_{1,n}}\right)\geq \frac{1}{2}\PP\left(M_1^2 \geq \frac{t+4+\frac{c}{2}}{\frac{c}{2}}\right)\geq\frac{1}{2}e^{-\beta t},
		\end{align*}
		where $\beta > 0$ is some constant, which depends on $c$.
		The first inequality uses the independence between $\{M_i\}$, the second inequality uses Chebyschev's inequality for $\nt^{(2)}(\mfN) - \lambda_{1,n}(M_1^2-1)$, which has $\mathrm{Var}(|\nt^{(2)}(\mfN) - M_1^2|) = \sum\limits_{i=2}^{\infty} \lambda_{i,n}^2 \leq 1$, and the fourth inequality is a tail bound for the chi-squared random variable $M_1^2$.
		But now, by~\cref{lem:Hdecomp}, for any $\rho$, $\nt^{(0, \rho)} + \nt^{(1, \rho)}(\mfN) \sim \mcN(\nt^{(0, \rho)}, \sigma^2_\rho)$ with $\sigma^2_\rho \leq 1$, and $\nt^{(0, \rho)} \leq m^2$. Thus, we have for some constant $\alpha > 0$, independent of $\rho$, and for all large $t > 10$,
		$$\PP\left(|\nt^{(0,\rho)} + \nt^{(1,\rho)}(\mfN)| \geq t\right) \leq e^{-\alpha t^2}.$$
		So,
		\begin{align*}
			\PP\left(|\nt^{(\rho)}(\mfN)| > t\right) &\geq \PP\left(|\nt^{(2)}(\mfN)| \geq 2t, |\nt^{(0, \rho)} + \nt^{(1, \rho)}(\mfN)| \leq t\right) \\
			&\geq \PP\left(|\nt^{(2)}(\mfN)| \geq 2t\right) + \PP\left(|\nt^{(0, \rho)} + \nt^{(1, \rho)}(\mfN)| \leq t\right) -1\\
			&\geq \frac{1}{2}e^{-\beta t} + 1- e^{-\alpha t^2} -1 = \frac{1}{2}e^{-\beta t} - e^{-\alpha t^2} \geq e^{-c't},
		\end{align*}
		for some $c' > 0$.
	\end{proof}
	We also need to handle the case where $\nt^{(2)}(\mfN)$ is not sub-exponential. In that case, as we shall show that $\nt^{(2)}(\mfN)$, and hence $\nt^{(\rho)}(\mfN)$, is approximately Gaussian, which makes $\nt(X)$ approximately a Gaussian mixture. When there is an influential edge this mixture will be non-trivial, thus not a Gaussian.
	\begin{lemma} \label{lem:Hmixt}
		Let $\{G_n\}_{n\geq 0}$ be a sequence of graphs with $\mathcal{I}_n$ the set of their influential vertices. Further let $\lambda_{1,n}$ be as in~\cref{lem:Hdecomp}, and assume
		$$\lim\limits_{n\to \infty}|{\lambda_{1,n}}| = 0.$$
		Then, there exists a subsequence $n_k$, such that for every partial coloring $\rho :\mathcal{I}_n \to \{-1,1\}$, 
		$\nt^{(\rho)}(\mfN) \xrightarrow[\mathrm{in\ law}]{n_k\to \infty}\mcN(\nt^{(0, \rho)}, \sigma'^2_\rho),$
		for some $\sigma'^2_\rho,$ and $\mfN \sim \mcN(0,\mathrm{I}_{|V(G_n) \setminus \mathcal{I}_n|})$.
	\end{lemma}
	\begin{proof}
		If $\mathrm{Var}(\nt^{(2)}(\mfN))\xrightarrow{n\to 0} 0$, then, by~\cref{lem:Hdecomp}
		$$\nt^{(\rho)}(\mfN) \xrightarrow[\mathrm{in\ law}]{n\to \infty} \nt^{(0, \rho)} + \nt^{(1, \rho)}(\mfN) \sim \mcN(0,\sigma_\rho^2).$$
		Thus, we suppose that $\limsup\limits_{n\to\infty} \mathrm{Var}(\nt^{(2)}(\mfN)) = \eta^2 > 0$.
		By~\cref{lem:Hdecomp}, we can write $\nt^{(2)}(\mfN)$ as $\sum\limits_{i=1}^\infty \lambda_{i,n}(M_i^2 - 1),$ and observe that by the Berry-Eseen inequality \cite{berry1941accuracy}
		$$d_{\mathrm{kol}}(\nt^{(2)}(\mfN),\widetilde \mcZ) \leq \max\limits_i \frac{\EE\left[(\lambda_{i,n}(M_i^2-1))^3\right]}{\EE\left[(\lambda_{i,n}(M_i^2-1))^2\right]}\frac{1}{\sqrt{\mathrm{Var}(\nt^{(2)}(\mfN))}} \leq \frac{\lambda_{1,n}}{\sqrt{\mathrm{Var}(\nt^{(2)}(\mfN))}},$$
		where we have passed to a subsequence and where $\widetilde \mcZ \sim \mcN(0,\eta^2)$.
		This implies that 
		$$\nt^{(2)}(\mfN) \xrightarrow[\mathrm{in\ law}]{n_k\to \infty}\mcN(0,\eta^2).$$
		Moreover, since $\nt^{(2)}$ is a homogeneous quadratic polynomial, $\nt^{(2)}(\mfN)$ belongs to the second Wiener chaos and its limit is independent of $\mfN$, see~\cite[Theorem 3.1]{nourdin2012convergence}. Thus, by~\cref{lem:Hdecomp}
		$$\nt^{(\rho)}(\mfN) \xrightarrow[\mathrm{in\ law}]{n_k\to \infty} \mcN(0,\sigma_\rho^2 + \eta^2).$$ \qedhere
	\end{proof}
	Having covered the two cases above we can now prove~\cref{prop:edgenoCLT}.
	\begin{proof}[Proof of~\cref{prop:edgenoCLT}]
		We separate the proof into two cases, depending on the behavior of $\nt^{(2)}(\mfN)$, as in ~\cref{lem:Hdecomp}. Recall that $|\mcI_n| \le m$ for some choice of $m$ uniformly bounded in $n$.
		
		{\bf Case I: } Suppose that $\limsup_{n\to \infty} |\lambda_{i,n}| > 0$. By passing to a subsequence, and applying~\cref{lem:Hsubexp}, we obtain  a constant $c' > 0$, such that for large enough $t > 0$,
		$$\PP\left(|\nt^{(\rho)}(\mfN)| > t\right) \geq e^{-c't},$$ 
		uniformly in $\rho.$  Thus, let $X\sim \mathrm{Rad}(\frac{1}{2})^{\otimes v(G_n)}$, so that $Z(\triangle, G_n) = \nt(X)$.
		By the law of total probability,
		$$\PP\left(|\nt(X)| > t\right) \ge \frac{1}{2^m}\sum\limits_{\rho: \mcI_n \to \{\pm1\}}\PP\left(|\nt^{(\rho)}(X)| > t\right)\xrightarrow{n\to\infty}\frac{1}{2^m}\sum\limits_{\rho: \mcI_n \to \{\pm1\}}\PP\left(|\nt^{(\rho)}(\mfN)| > t\right) \gtrsim e^{-c't}.$$
		The final limit follows by observing that for a given partial assignment $\rho$ of all influential vertices, per~\cref{lem:countinginfluence}, $\nt^{(\rho)}$ has no influential variables and thus~\cref{thm:invariance} implies convergence.
		Since $\PP\left(|\mcZ| \geq t\right) \leq 2e^{-t^2/2}$ for $\mcZ \sim \mcN(0, 1)$, we conclude the proof by choosing an appropriately large $t$.
		
		{\bf Case II: }
		Suppose that $\lim_{n\to \infty} |\lambda_{i,n}| = 0$, and let $X \sim \mathrm{Rad}(\frac{1}{2})^{\otimes V(G_n)}$ be as above. 
		Let $Y$ be a Gaussian mixture with law $$Y \overset{\text{law}}=\frac{1}{2^{|\mcI_n|}}\sum\limits_{\rho : \mcI_n \to \{\pm 1\}} \mcN(\nt^{(0, \rho)}, \sigma'^2_\rho),$$ where the variances 
		$\sigma'^2_\rho$ are as in~\cref{lem:Hmixt}. We claim that $\nt(X)$ and $Y$ are asymptotically the same.
		Indeed, by the law of total probability,
		\begin{align*}
			d_{\mathrm{kol}}(\nt(X),Y) &\leq  d_{\mathrm{kol}}(\nt(X),\nt(\mfN)) + d_{\mathrm{kol}}(\nt(\mfN),Y) \\
			& \leq \frac{1}{2^{|\mcI_n|}}\sum\limits_{\rho}d_{\mathrm{kol}}(\nt^{(\rho)}(X),\nt^{(\rho)}(\mfN)) + d_{\mathrm{kol}}(\nt(\mfN),Y).
		\end{align*}
		Since $\nt^{(\rho)}$ has no influential variables,~\cref{thm:invariance} implies 
		$$\frac{1}{2^{|\mcI_n|}}\sum\limits_{\rho}d_{\mathrm{kol}}(\nt^{(\rho)}(X),\nt^{(\rho)}(\mfN))\xrightarrow{n\to \infty} 0,$$
		while the condition on $\lambda_{1,n}$ and~\cref{lem:Hmixt} ensure that
		$$d_{\mathrm{kol}}(\nt(\mfN),Y)\xrightarrow{n\to \infty} 0.$$
		Thus, $d_{\mathrm{kol}}(\nt(X),Y)\xrightarrow{n\to \infty} 0.$ and it will be enough to show, that when $G_n$ contains an influential edge, then
		$d_{\mathrm{kol}}(Y,\mcZ)$ is bounded away from $0$ (where $\mcZ \sim \mcN(0, 1)$). Since $Y$ is a finite Gaussian mixture, and since the family of finite Gaussian mixtures is identifiable, \cite{teicher1963identifiability}, for $Y \xrightarrow[\text{in law}]{n\to\infty} \mcZ$ it is necessary and sufficient that, for every $\rho$,
		$$\nt^{(0, \rho)} \xrightarrow{n\to \infty} 0\text{ and } \sigma_\rho' \xrightarrow{n\to \infty} 1.$$
		The third item in~\cref{lem:Hdecomp} along with the existence of an influential edge rules out the possibility of $\nt^{(0, \rho)} \xrightarrow{n\to \infty} 0$ for every $\rho$. Hence $d_{\mathrm{kol}}(Y,\mcZ)$ must remain bounded away from $0$.
	\end{proof}
	\subsection{Revisiting the Fourth-Moment Phenomenon for $Z(\triangle, G_n)$}
	We now prove the second item of \cref{thm:inf-edge} and show that when $G_n$ does not contain an influential edge, then the fourth-moment phenomenon persists.
	\begin{prop}\label{p:lowinftriangle}
		Let ${G_n}$ be a sequence of graphs and set $M_n:= \max\limits_{e\in E(G)}\frac{D_e(\triangle,H)}{\sigma_\triangle}$. Then, for $\mcZ \sim \mcN(0,1)$,
		$$d_{\mathrm{kol}}(Z(\triangle,G_n),\mcZ) \lesssim \left(M_n^2  + \EE\left[Z(\triangle, G_n)^4\right] - 3 \right)^{1/5}.$$
	\end{prop}
	\begin{proof}
		We shall require several results from \cite{BHA22} concerning the fourth moment of $Z(\triangle,G_n)$, and its distance to a Gaussian.
		First, \cite[Theorem 1.1, Lemma 4.1]{BHA22} established that 
		\begin{equation} \label{eq:trianglekolm}
			d_{\mathrm{kol}}(Z(\triangle,G_n),\mcZ) \lesssim \left(\left(\frac{1+N(\triangle_4,G_n)}{\sigma_{\triangle}^4}\right)^{\frac{1}{4}}+\frac{\sum\limits_{s=1}^4N(\triangle_s, G_n) + \sum\limits_{s=1}^{28}N(H_s,G_n)}{\sigma^4_\triangle} \right)^{1/5}.
		\end{equation}
		Here $\triangle_s$ stands for an $s$-pyramid ($s$ triangles glued at a common edge) and $\{H_s\}_{s=1}^{28}$ is a list of $28$ graphs given in~\cite[Lemma 4.1]{BHA22}. In this lemma, the authors also show that there is a list of positive constants $\{\delta_s\}_{s=1}^4$ and $\{c_s\}_{s=1}^{28}$, such that 
		\begin{equation} \label{eq:trianglesfourth}
			\EE\left[Z(\triangle, G_n)^4\right] - 3 =\frac{1}{\sigma_\triangle^4}\left[-\sum\limits_{s=1}^4\delta_sN(\triangle_s, G_n) + \sum\limits_{s=1}^{28}c_sN(H_s,G_n)\right].
		\end{equation}
		Our key observation is that we can enumerate the $s$-pyramids by enumerating over edges $e\in E(G_n)$, and choosing $s$ different triangles which contain $e$. Recall that for $e\in E(G_n)$, there are precisely $D_e(\triangle,G_n)$ triangles which include $e$, and so
		$$
		N(\triangle_s,G_n) \leq \sum\limits_{e\in E(G_n)}D_e(\triangle,G_n)^s.$$
		We thus have, for every $1 \le s \le 4$,
		\begin{align*}
			\frac{N(\triangle_s, G_n)}{\sigma^4_\triangle} &\le \frac{\sum\limits_{e\in E(G_n)}D_e(\triangle,G_n)^4}{\sigma_{\triangle}^4} \le M_n^2 \cdot \frac{\sum\limits_{e\in E(G_n)}D_e(\triangle,G_n)^2}{\sigma_{\triangle}^2} \le 10 M_n^2,
		\end{align*}   
		where in the final inequality, we used the expression for $\sigma_{\triangle}^2$ given in~\cref{eq:varrep}.
		Plugging this into~\cref{eq:trianglesfourth} we get, for some constant $C > 0$,
		$$-C M_n^2 + \frac{\sum\limits_{s=1}^{28}N(H_s,G_n)}{C\sigma_\triangle^4}\leq \EE\left[Z(\triangle, G_n)^4\right] - 3.$$
		Applying the same bound to~\cref{eq:trianglekolm} yields that, 
		$$ d_{\mathrm{kol}}(Z(\triangle,G_n),\mcZ) \lesssim \left(M_n^{\frac{1}{2}}+ \frac{\sum\limits_{s=1}^{28}N(H_s,G_n)}{\sigma^4_\triangle}\right)^{1/5}.$$
		With the above two inequalities, we can now conclude with
		\begin{align*}
			d_{\mathrm{kol}}(Z(\triangle,G_n),\mcZ) \lesssim \left(M_n^{\frac{1}{2}} + \left(\EE\left[Z(\triangle, G_n)^4\right] - 3\right) \right)^{1/5}. 
		\end{align*}
	\end{proof}
 
	\section{Towards Characterizing Asymptotic Normality of General Monochromatic Subgraph Counts}
	In this section, we generalize some of the ideas presented in \cref{sec:triangles} to more general monochromatic subgraph counts and give evidence towards \cref{c:main}.
	\subsection{Graph Sequences Without Influential Pairs Exhibit Fourth-Moment Phenomena}
	We begin by generalizing \cref{p:lowinftriangle} to the general case. We shall prove \cref{thm:inf-edge-general} and show that the fourth-moment phenomenon holds whenever the sequence $G_n$ does not contain an influential pair. For convenience, we recall the statement of the theorem. 
	\begin{thm}\label{t:conj-one-dir}
		Consider a sequence of graphs $\{G_n\}$, a graph $H$, and let $\sigma_{H} = \sqrt{\mathrm{Var}(T(H, G_n))}.$ If $M_n := \max\limits_{e \in E(G_n)} \frac{D_e(\triangle, G_n)}{\sigma_{H}},$ then, for $\mcZ \sim \mcN(0, 1)$, 
		$$d_{\mathrm{kol}}(Z(H,G_n),\mcZ) \lesssim \left(M_n^2  + (\EE\left[Z(H, G_n)^4\right] - 3) \right)^{\frac{1}{20}}.$$
    \end{thm}

	In~\cref{p:lowinftriangle}, we showed that the distance of $Z(\triangle, G_n)$ to a standard Gaussian $\mcZ$ was controlled by the maximum influence of a pair of vertices and the divergence of the fourth moment from $3 = \EE[\mcZ^4]$. The key idea in the argument was observing that the upper bound of \cite[Theorem 1.1, Lemma 4.1]{BHA22} on $d_{\mathrm{kol}}(Z(\triangle,G_n),\mcZ)$ was given by a sum over various connected \textit{$4$-joins} of triangles (ways to glue $4$ triangles together into a connected graph). These $4$-joins were observed to be a subset of those which appeared with some coefficients in the fourth-moment difference $\EE[Z(\triangle, G_n)^4] - 3$, and the only $4$-joins with negative coefficients in this fourth-moment difference had counts controlled by the influence of edges.
	
	The work of~\cite{BHA22} was extended to general $H$ in~\cite{DHM22}, enabling us to analogously extend~\cref{p:lowinftriangle} to general $H$. Before stating the needed result from~\cite{DHM22}, we begin first define the set of $4$-joins appearing in the error term.
	
	\begin{defn}[Good join] \label{def:good}  Given a simple connected graph $H$, a \emph{good join} of $H$ is a graph formed by joining $4$ copies of $H$, say $H_1,H_2,H_3,H_4$, such that for all  $i=1,2,3,4$
		\begin{align} \label{good}
			\bigg| V(H_1\cup H_2\cup H_3 \cup H_4)\bigg|-\bigg| V\big(\bigcup_{j \neq i} H_j\big)\bigg| \le |V(H)|-2.
		\end{align}
		and
		\begin{align}\label{vgood}
			|V(H_1\cup H_2\cup H_3 \cup H_4)| \le \min_{\pi\in \mathbb{S}_4} \bigg\{ |V(H_{\pi_1}\cup H_{\pi_2})|+|V(H_{\pi_3}\cup H_{\pi_4})|-2 \bigg\}.
		\end{align}
		In the expression above, we take the minimum over $\pi \in \mathbb{S}_4$, the set of permutations of $(1,2,3,4)$. We denote the collection of all good joins of $H$ by $\gj$, and by $N(\gj,G_n)$ the number of different good joins in $G_n$.
	\end{defn}
	\begin{thm}\label{t:main-dhm}\textnormal{(\cite[Theorem 1.4]{DHM22})}
		Fix a connected graph $H$ with $v(H)=r\ge 3$ and consider $Z(H,G_n)$ defined in \eqref{d:zhgn}. Then, we have that
		$$d_{\mathrm{kol}}(Z(H,G_n),\mcZ) \lesssim \left(\frac{N(\gj,G_n)}{\sigma_H(G_n)^4}\right)^{1/20} $$
	\end{thm}
	We will also use the following result (contained in the $c=2$ case of ~\cite[Proposition 4.3]{DHM22}).
	
	\begin{prop}[\cite{DHM22}]
		Fix connected graph $H$ with $v(H) = r$ and consider a graph sequence $\{G_n\}$. For $\s \in V(G_n)_r$,
		let $$W_{\s}:=\frac{1}{|\Aut(H)|}a_{H, \bs}(\ind\{1_{X_{=\s}}\}-{2}^{1-r})=\frac{1}{|\Aut(H)|}a_{H, \bs}Z_{\bs},$$
		where $Z_{\bs}:=\ind_{X_{=\s}}-2^{1-r},$ and observe that $Z(H,G_n)=\frac1{\sigma_H(G_n)}\sum_{\s} W_s.$ Then,
  \begin{align*}
      \EE[Z(H,G_n)^4]-3 &=  \frac1{\sigma_H(G_n)^4}\sum_{\{\s_1,\s_2,\s_3,\s_4\} \in \mcG_H} \EE[W_{\s_1}W_{\s_2}W_{\s_3}W_{\s_4}]\\
      &-\frac3{\sigma_H(G_n)^4} \underbrace{\sum_{ |\s_1 \cap \s_2 | \ge 2} \sum_{|\s_3 \cap \s_4 |\ge 2 }}_{\{ \s_1 \bigcup  \s_2\} \bigcap \{ | \s_3 \bigcup  \s_4\}|\ge 2} \EE[W_{\s_1}W_{\s_2}]\EE[W_{\s_3}W_{\s_4}].
  \end{align*}
		
	\end{prop}
	
	In order to prove~\cref{t:conj-one-dir}, we will need to understand more delicately which joins make positive and negative contributions to the fourth moment error above.
	
	\begin{lemma}\label{l:pie-4}
		For $\{\s_1, \s_2, \s_3, \s_4\} \in \mcG_H$, we have that
		$$    \EE[Z_{\s_1}Z_{\s_2}Z_{\s_3} Z_{\s_4}] =- 2^{5 - 4r} + 2^{3-2r}\sum_{i < j} 2^{- |\s_i \cup \s_j|} - 2^{2-r} \sum_{i = 1}^4 2^{- |\cup_{j \neq i} \s_j|} + 2^{1 - |\cup_i \s_i|} $$
	\end{lemma}
	\begin{proof}
		This follows by the principle of inclusion-exclusion:
		\begin{align*}
			\EE[Z_{\s_1}Z_{\s_2}Z_{\s_3} Z_{\s_4}] & = \sum_{A \subset [4]} (-1)^{4 - |A|} \EE\left[ 2^{(1-r)(4-|A|)}\prod_{j\in A} \ind\{X_{=\s_{j}}\}\right] \\ & = \sum_{A \subset [4]} (-1)^{4 - |A|} 2^{(1-r)(4-|A|)+1-|\bigcup_{j\in A} {\s}_{j}|}. \\
			&= 2^{4(1-r) + 1} - 4 \cdot 2^{3(1 - r) + 1 - r} + \sum_{i < j} 2^{2(1-r) + 1 - |\s_i \cup \s_j|} - \sum_{i = 1}^4 2^{(1-r) + 1 - |\cup_{j \neq i} \s_j|} + 2^{1 - |\cup_i \s_i|} \\
			&= - 2^{5 - 4r} + 2^{3-2r}\sum_{i < j} 2^{- |\s_i \cup \s_j|} - 2^{2-r} \sum_{i = 1}^4 2^{- |\cup_{j \neq i} \s_j|} + 2^{1 - |\cup_i \s_i|}.
		\end{align*}
	\end{proof}
	
	Similarly to~\cref{eq:varrep}, we can more generally express $\Var(T(H, G_n))$ in terms of total influences of subsets of vertices, as computed explicitly in~\cite{DHM22}, and is the analog of \eqref{eq:varrep} for more general shapes.
	
	\begin{obs}\label{eq:var-gen}
		For connected graph $H$ with $v(H) = r,$ we have that
		\begin{equation}
			\sigma_H(G_n)^2 = \Var(T(H, G_n)) = \Theta \left(\sum_{k = 2}^r \sum_{\rb \in V(G_n)_k} D_{\rb}(H, G_n)^2 \right) = \Theta\left( \sum_{\rb \in V(G_n)_2} D_{\rb}(H, G_n)^2 \right)
		\end{equation}    
	\end{obs}
	
	With these pieces, we can show that if $G_n$ lacks influential pairs and the fourth moment of $Z(H, G_n)$ is asymptotically $3$, then $T(H, G_n)$ is asymptotically Gaussian.
	
	\begin{proof}[Proof of~\cref{t:conj-one-dir}]
		
		We consider two cases on $\{\s_1, \s_2, \s_3, \s_4\} \in \mcG_H$, letting the corresponding $4$-join by this $4$-tuple be the graph $H'$.
		\begin{enumerate}
			\item Suppose that for some $i \neq j$ $|\s_i \cap \s_j| \ge 2$ (without loss of generality take $i = 1, j = 2$). Then, by summing over pairs of vertices $\w_{12} \subset \s_1 \cap \s_2$, we get 
			\footnotesize 
			$$\frac{N(H', G_n)}{\sigma_H(G_n)^4} \le \frac{1}{\sigma_H(G_n)^4}\sum\limits_{\substack{\rb_{12}, \rb_3, \rb_4 \in V(G_n)_2}} D_{\rb_{12}}(H, G_n)^2 D_{\rb_3}(H, G_n) D_{\rb_4}(H, G_n) \lesssim M_n^2 \frac{ \sum\limits_{\rb} D_{\rb}(H, G_n)^2}{\sigma_H(G_n)^2} \lesssim M_n^2,$$
			\normalsize 
			where the final inequality follows by~\cref{eq:var-gen}. 
			\item Else $\{\s_1, \s_2, \s_3, \s_4\} \in \mcG_H$ forms a good join $H'$, where for every $i \neq j \in [4]$, $|\s_i \cap \s_j| \le 1$. The only contribution of such joins to $\EE[Z(H, G_n)^4] - 3$ is given by $\frac1{\sigma_H^4} \EE[W_{\s_1}W_{\s_2}W_{\s_3}W_{\s_4}]$. Since every good join is connected, we must have $v(H') \le r + 3(r-1) = 4r - 3$. Further since this good join has no pair of copies of $H$ that share more than $1$ vertex, we additionally have $v(H') \ge 4r - 6$. In order for $\{\s_1, \s_2, \s_3, \s_4\}$ to be a good join, for at least $4$ of the $6$ pairs $\{\s_i, \s_j\}_{i \neq j}$, we must have $|\s_i \cap \s_j| = 1$. By~\cref{l:pie-4}, we have that 
			\begin{align*}
				\EE[Z_{\s_1}Z_{\s_2}Z_{\s_3} Z_{\s_4}] &=- 2^{5 - 4r} + 2^{3-2r}\sum_{i < j} 2^{- |\s_i \cup \s_j|} - 2^{2-r} \sum_{i = 1}^4 2^{- |\cup_{j \neq i} \s_j|} + 2^{1 - |\cup_i \s_i|} \\
				&\overset{(*)}\ge 2^{5 - 4 r} + 2^{3 - 2 r} \cdot (4 \cdot 2^{-(2 r - 1)} + 2\cdot 2^{-2 r}) - 2^{2 - r}\cdot  (4 \cdot 2^{-(v(H') - (r - 2))}) + 2^{1 - v(H')} \\
				&= 7 \cdot 2^{4 - 4r} - 2^{1 - v(H')},
			\end{align*}
			where in (*), we used that $|\s_i \cap \s_j| = 1$ for at least $4$ pairs $\{\s_i, \s_j\}_{i \neq j}$, and that since $\{\s_1, \s_2, \s_3, \s_4\}$ forms a good join, for any $i \in [4],$ $|\cup_{j \neq i} \s_j| \ge v(H') - (r-2)$. Since $7 \cdot 2^{4 - 4r} - 2^{1 - v(H')} > 0$ for $v(H') \ge 4r - 5$, we turn our attention to the case where $v(H') = 4r - 6$. In this case,  we necessarily have $|\s_i \cap \s_j| = 1$ for all $6$ pairs, and the above bound tightens to
			\begin{align*}
				\EE[Z_{\s_1}Z_{\s_2}Z_{\s_3} Z_{\s_4}] &=- 2^{5 - 4r} + 2^{3-2r}\sum_{i < j} 2^{- |\s_i \cup \s_j|} - 2^{2-r} \sum_{i = 1}^4 2^{- |\cup_{j \neq i} \s_j|} + 2^{1 - |\cup_i \s_i|} \\
				&\ge 2^{5 - 4 r} + 2^{3 - 2 r} \cdot (6 \cdot 2^{-(2 r - 1)}) - 2^{2 - r}\cdot  (4 \cdot 2^{-(v(H') - (r - 2))}) + 2^{1 - v(H')} \\
				&= 0.
			\end{align*}
			The above estimates imply that for whenever $\{\s_1, \s_2, \s_3, \s_4\} \in \mcG_H$ forms a good join $H'$, where for every $i \neq j \in [4]$, $|\s_i \cap \s_j| \le 1$, we have $\EE[Z_{\s_1}Z_{\s_2}Z_{\s_3} Z_{\s_4}] \ge 0$.
		\end{enumerate}
		Let $\widetilde\mcG_H \subset \mcG_H$ be the tuples $\{\s_1, \s_2, \s_3, \s_4\} \in \mcG_H$ that form a good join $H'$, where for every $i \neq j \in [4]$, $|\s_i \cap \s_j| \le 1$. Putting the pieces together from the above two cases, we can conclude that
		$$\EE[Z(H, G_n)^4] - 3 \gtrsim - M_n^2 + \frac{1}{\sigma_H(G_n)^4}\sum_{H' \in \widetilde\mcG_H} N(H', G_n).$$
		Combining the above inequality with~\cref{t:main-dhm} gives the desired result:
		\begin{align*}
			d_{\mathrm{kol}}(Z(H,G_n),\mcZ) &\lesssim \left(\frac{N(\gj,G_n)}{\sigma_H(G_n)^4}\right)^{1/20} \\
			&\lesssim \left( M_n^2 + \left(\EE[Z(H, G_n)^4] - 3\right) \right) ^{1/20}.
		\end{align*}
	\end{proof}
	
	\subsection{Strongly Influential Pairs Forbid CLTs} \label{sec:strong}
	We now give some evidence towards~\cref{c:main}, by showing the conjecture holds for a stronger notion of influence. Namely, we show that if $\{G_n\}$ has \textit{strongly influential pairs}, those forbid a CLT.
	\begin{defn}[Strongly influential vertices]
		A $k$-tuple of vertices $\rb \in V(G_n)_k$ is $\eps$-\textit{strongly influential}, if $D_{\rb}(H, G_n) \ge \eps N(H, G_n).$
	\end{defn}
	It is readily seen that $N(H,G_n) \gtrsim \sigma_H(G_n)$, which justifies calling the tuple in the above definition \emph{strongly} influential. Below we observe that when $G_n$ contains an influential pair, then $Z(H,G_n)$ will be bounded almost surely, \emph{uniformly in $n$}. Clearly, $Z(H,G_n)$ is not a Gaussian in this case.
	\begin{prop} \label{prop:stronginf}
		Consider a fixed $H$ and graph sequence $\{G_n\}$. If for all $N$, there exists some $n > N$ such that $G_n$ has a strongly influential pair of vertices $\rb = (v, w) \in V(G_n)_2$, then $Z(H, G_n)  \centernot{\xrightarrow{\text{in law}}} \mcN(0, 1)$.
	\end{prop}
	\begin{proof}
		Fix connected simple graph $H$ on $r$ vertices. Consider graph  $G_n$ with strongly influential vertex pair $\mathbf{z} \in V(G_n)_2,$ so that $D_{\bz}(H, G_n) \gtrsim T(H, G_n).$
		Per~\cref{eq:var-gen}, we can lower bound the variance as 
		$$\sigma_{H}(G_n)^2 \gtrsim D_{\bz}(H, G_n)^2 \gtrsim N(H, G_n)^2$$
		However, since $N(H, G_n) \ge T(H, G_n)$, this implies that
		$$Z(H, G_n) = \frac{T(H, G_n) - \EE T(H, G_n)}{\sigma_H(G_n)} \le \frac{T(H, G_n)}{\sigma_H(G_n)} = O(1)$$
		is bounded almost surely. Therefore, $Z(H, G_n)  \centernot{\xrightarrow{\text{in law}}} \mcN(0, 1)$.
	\end{proof}
    \begin{rmk}
    There are some other more specific cases where influential pairs can forbid a CLT. For example, suppose that $H = K_4$, so that as in \cref{eq:countingmultirep}, $\MP = F_{H,G,0} + F_{H,G,2} + F_{H,G,4}$.
    The calculations in \cref{lem:homogenvar} show that 
    $$\EE\left[F_{H,G,4}^2\right] \leq N(H,G) \text{ and } \EE\left[F^2_{H,G,2}\right] \gtrsim \max\limits_{v\in V(G)} D_v^2(H,G).$$
    Thus, let $\{G_n\}$ be a sequence of graphs satisfying $D_v(H,G_n) \gg \sqrt{N(H,G_n)}$, and set $\widetilde{\MP}$ as in \cref{lem:polyCLT}. Then, $\widetilde{F_{H,G_n}} = \widetilde{F_{H,G_n,2}} + o(1)$ is asymptotically a quadratic function, and the assertions of \cref{sec:triangles} hold in this case as well. Note that the requirement $D_v(H,G_n) \gg \sqrt{N(H,G_n)} $ is much weaker than having a strongly influential pair, yet still stronger than $D_v(H,G_n) = \Omega\left(\sigma_H(G_n)\right)$.
    \end{rmk}

\subsection{Aligning Higher Moments with Influential Edges}\label{s:exbad}
In~\cref{prop:edgenoCLT}, we showed that the existence of an influential edge precludes convergence of $Z(\triangle, G_n)$ to a standard Gaussian $\mcZ$. Our approach was to study the tail behavior of an associated quadratic form in Gaussian variables and to apply the identifiability of Gaussian mixtures. For more general $H$, per~\cref{c:main}, we expect influential pairs of vertices to similarly preclude a central limit theorem. One very natural strategy to try to exclude a CLT is to study higher moments of $Z(H, G_n)$ and demonstrate that e.g. if $\EE[Z(H, G_n)^4] \simeq 3$, then the $6$th (or higher) moment must disagree with that of a standard Gaussian. Below, we give an example that illustrates a limitation of this type of plan by exhibiting a sequence of graphs where the first $6$ moments of $Z(\triangle, G_n)$ all line up with a standard Gaussian, but per~\cref{prop:edgenoCLT} we can see that a CLT fails to hold.

We give a sequence of graphs $\{G_n\}$ where $\EE[Z(\triangle, G_n)^4] \to 3, \EE[Z(\triangle, G_n)^6] \to 15$, but $Z(\triangle, G_n)$ can be shown to not be asymptotically normal.
We let $G_n = S_n \sqcup P_n \sqcup B_n$ be the disjoint union of three graphs that we define below:
\begin{align*}
	V(S_n) &= \{v_k\}_{k \in [4]} \cup \left\{ t^{(k)}_i\right\}_{i \in [n], k \in [4]}, \\
	&\qquad E(S_n) = \left\{ \left(v_k, v_{k+1 (4)}\right)\}_{k \in [4]} \cup \{ \left(v_{k}, t^{(k)}_i\right), \left(t^{(k)}_i, v_{k+1 (4)} \right) \right\}_{i \in [n], k \in [4]} \\
	V(P_n) &= \{w, x\} \cup \left\{ p_i \right\}_{i \in [cn]}, \\
	&\qquad E(P_n) = \left\{ \left(w, x\right)\} \cup \{ \left(w, p_i\right), \left(p_i, x \right) \right\}_{i \in [cn]} \\
	V(B_n) &= \{y_i\}_{i \in [bn^2]} \cup \left\{ \ell^{(ij)}_0, \ell^{(ij)}_i, \ell^{(ij)}_j, r^{(ij)}_0, r^{(ij)}_i, r^{(ij)}_j\right\}_{i < j \in [bn^2]} \\
	&\qquad E(B_n) = \left\{ (y_i, \ell_0^{(ij)}), (y_i, \ell_i^{(ij)}),(\ell_i^{(ij)}, \ell_0^{(ij)}), (y_j, \ell_0^{(ij)}), (y_j, \ell_j^{(ij)}),(\ell_j^{(ij)}, \ell_0^{(ij)})\right\}_{i < j \in [bn^2]} \\
	&\qquad \qquad \qquad \quad \cup \left\{ (y_i, r_0^{(ij)}), (y_i, r_i^{(ij)}),(r_i^{(ij)}, r_0^{(ij)}), (y_j, r_0^{(ij)}), (y_j, r_j^{(ij)}),(r_j^{(ij)}, r_0^{(ij)})\right\}_{i < j \in [bn^2]} 
\end{align*}
$G_n$ is pictured in~\cref{f:fig-gn}. 
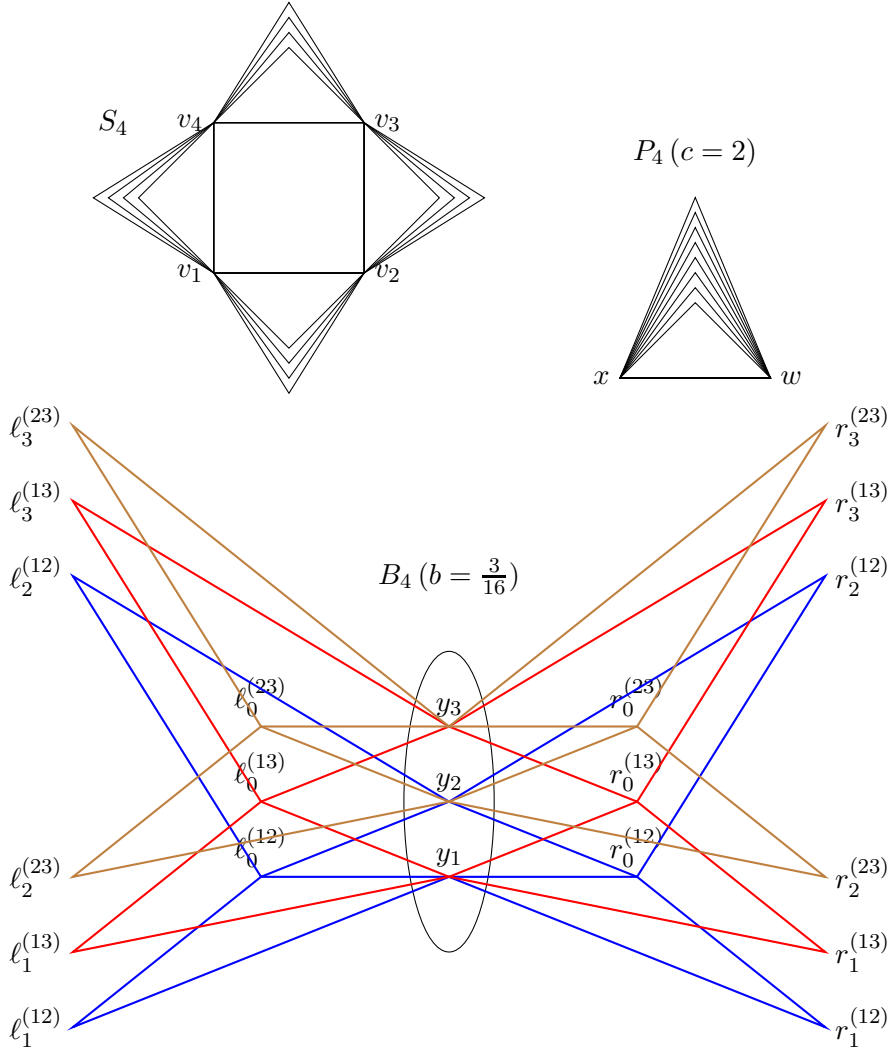
\begin{figure}
\begin{tikzpicture}
    \coordinate (v1) at (0,0);
    \coordinate (v2) at (2,0);
    \coordinate (v3) at (2,2);
    \coordinate (v4) at (0,2);
    
    \draw (v1) -- (v2) -- (v3) -- (v4) -- cycle;
    
    \node at (v1) [left] {$v_1$};
    \node at (v2) [right] {$v_2$};
    \node at (v3) [right] {$v_3$};
    \node at (v4) [left] {$v_4$};
    \node at (-1, 2) [left] {$S_4$};
    
    \draw (v1) -- (1,-1) -- (v2) -- cycle;
    \draw (v1) -- (1,-1.2) -- (v2) -- cycle;
    \draw (v1) -- (1,-1.4) -- (v2) -- cycle;
    \draw (v1) -- (1,-1.6) -- (v2) -- cycle;
    \draw (v2) -- (3,1) -- (v3) -- cycle;
    \draw (v2) -- (3.2,1) -- (v3) -- cycle;
    \draw (v2) -- (3.4,1) -- (v3) -- cycle;
    \draw (v2) -- (3.6,1) -- (v3) -- cycle;
    \draw (v3) -- (1,3) -- (v4) -- cycle;
    \draw (v3) -- (1,3.2) -- (v4) -- cycle;
    \draw (v3) -- (1,3.4) -- (v4) -- cycle;
    \draw (v3) -- (1,3.6) -- (v4) -- cycle;
    \draw (v1) -- (-1,1) -- (v4) -- cycle;
    \draw (v1) -- (-1.2,1) -- (v4) -- cycle;
    \draw (v1) -- (-1.4,1) -- (v4) -- cycle;
    \draw (v1) -- (-1.6,1) -- (v4) -- cycle;
\end{tikzpicture}
\hspace{30pt}
\begin{tikzpicture}

    \coordinate (v3) at (2,2);
    \coordinate (v4) at (0,2);
    
    \draw  (v3) -- (v4) ;
    
    \node at (v3) [right] {$w$};
    \node at (v4) [left] {$x$};
    \node at (1, 5) {$P_4\,(c = 2)$};
    
    \draw (v3) -- (1,3) -- (v4) -- cycle;
    \draw (v3) -- (1,3.2) -- (v4) -- cycle;
    \draw (v3) -- (1,3.4) -- (v4) -- cycle;
    \draw (v3) -- (1,3.6) -- (v4) -- cycle;
    \draw (v3) -- (1,3.8) -- (v4) -- cycle;
    \draw (v3) -- (1,4.0) -- (v4) -- cycle;
    \draw (v3) -- (1,4.2) -- (v4) -- cycle;
    \draw (v3) -- (1,4.4) -- (v4) -- cycle;
\end{tikzpicture}
\begin{tikzpicture}
    \foreach \i in {1,2,3} {
        \coordinate (y\i) at (0,\i);
        \node at (y\i) [above] {$y_{\i}$};
    }
    \node at (0, 5) {$B_4\, (b = \frac{3}{16})$};
    
    \draw[black, thin] (-0, 2) ellipse (0.6 and 2);
    \def\mycolors{{"blue", "red", "brown"}}
    \foreach \i in {1,2,3} {
        \foreach \j in {1,2,3} {
            \ifnum\j>\i
                \pgfmathtruncatemacro\result{\i+\j}
              \coordinate (a\i\j) at (2.5,\i + \j - 2);
                \coordinate (b\i\j) at (-2.5,\i + \j -2);
                \coordinate (z\i\j) at (5,\i + \j - 4);
                \coordinate (z\j\i) at (5,\j + \i + 2);
                \coordinate (w\i\j) at (-5,\i + \j - 4);
                \coordinate (w\j\i) at (-5,\j + \i + 2);
                \node at (a\i\j) [above] {$r_0^{(\i\j)}$};
                \node at (z\i\j) [right] {$r_{\i}^{(\i\j)}$};
                \node at (z\j\i) [right] {$r_{\j}^{(\i\j)}$};
                \node at (b\i\j) [above] {$\ell_0^{(\i\j)}$};
                \node at (w\i\j) [left] {$\ell_{\i}^{(\i\j)}$};
                \node at (w\j\i) [left] {$\ell_{\j}^{(\i\j)}$};
            \fi
        }
    }
    
    Draw triangles
    \foreach \i in {1,2, 3} {
     \foreach \j in {1,2,3} {
        \ifnum\j>\i
        \pgfmathtruncatemacro\result{\i+\j}
                    \pgfmathparse{\mycolors[\result-3]}
                \edef\mycolor{\pgfmathresult}  
            \draw[\mycolor, thick] (y\i) -- (a\i\j) -- (z\i\j) -- cycle;
            \draw[\mycolor, thick] (y\j) -- (a\i\j) -- (z\j\i) -- cycle;
            \draw[\mycolor, thick] (y\i) -- (b\i\j) -- (w\i\j) -- cycle;
            \draw[\mycolor, thick] (y\j) -- (b\i\j) -- (w\j\i) -- cycle;
         \fi
        }
    }
\end{tikzpicture}
\caption{An example of a graph $G_n$ with $n = 4, b = 3/16, c = 2$ as a disjoint union of graphs $S_4, P_4, B_4$ as pictured}
\label{f:fig-gn}
\end{figure}

Let $T'(\triangle, G_n) = 4(T(\triangle, G_n) - \EE[T(\triangle, G_n)])$ (centered but not normalized). Per~\cref{eq:varrep}, 
$$\EE[T'(\triangle, G_n)^2] = \sum_{e \in E(G)} D_e^2(\triangle, G).$$
We then compute the fourth moment by expanding the representation of $T(\triangle, G_n)$ of~\cref{d:thgn}, noticing that some terms vanish. Simplifying yields the following:
\begin{align*}
	\EE[T'(\triangle, G_n)^4] &= 3 \left(\sum_{e \in E(G_n)} D_e^2(\triangle, G_n) \right)^2 - 2 \sum_{e \in E(G_n)} D_e^4(\triangle, G_n) \\
	&\qquad + 24 \sum_{C_4 \in \mcC_4(G_n)} \prod_{e \in E(C_4)} D_{e}(\triangle, G_n) + o(\sigma_{\triangle}^4), \intertext{where the latter term is a sum over all of the $4$-cycle's in $G$, requiring that no two $4$-cycles in the sum have the same edge set. We similarly compute the sixth moment as follows}
	\EE[T'(\triangle, G_n)^6] &= 720 \sum_{C_6 \in \mcC_6(G_n)} \prod_{e \in E(C_6)} D_e(\triangle, G_n) - 240 \sum_{C_4 \in \mcC_4(G_n)} \left(\prod_{e \in C_4} D_e(\triangle G_n) \right) \left(\sum_{e \in C_4} D_e(\triangle, G_n)^2 \right) \\
	&\qquad + 15 \left( \sum_{e \in E(G_n)} D_e(\triangle, G_n)^2 \right)^3 + 16 \sum_{e \in E(G_n)} D_e(\triangle, G_n)^6 + o(\sigma_{\triangle}^6).
\end{align*}
When we specialize in the sequence of graphs chosen above, we obtain the following moments.
\begin{align*}
	\EE[T'(\triangle, G_n)^4] &= 3 \EE[T'(\triangle, G_n)^2]^2 - 2(4n^4 + (cn)^4) + 24\left(n^4 + \binom{bn^2}{2}\right) + o(n^4) \\
	\EE[T'(\triangle, G_n)^6] &= -240 \cdot 4 n^6 + 15 \EE[T'(\triangle, G_n)^2]^3 + 16 (4n^6 + (cn)^6 ) + o(n^6)
\end{align*}
Then, if we choose $b = \sqrt{\frac{2}{3} \left(7^{2/3}-2\right)} \approx 1.0518, c = \sqrt{2} \cdot 7^{1/6} \approx 1.956$, we find that as $n \to \infty$, $\EE[T'(\triangle, G_n)^4] \longrightarrow 3 \EE[T'(\triangle, G_n)^2]^2 $ and $\EE[T'(\triangle, G_n)^6] \longrightarrow 15 \EE[T'(\triangle, G_n)^2]^3 $, meaning that the normalized versions satisfy $\EE[Z(\triangle, G_n)^4] \to 3, \EE[Z(\triangle, G_n)^6] \to 15$.

However, we can see that $T(\triangle, G_n)$ is not asymptotically normal, by noting that $T(\triangle, G_n) = T(\triangle, S_n) + T(\triangle, P_n) + T(\triangle, B_n)$ is the sum of three independent (in fact with disjoint vertex support) random variables. Note that by considering whether or not $w, x \in V(P_n)$ receive the same or different colors, we see that the random variable $T(\triangle, P_n)$ is equivalent in distribution to
$$T(\triangle, P_n) \equiv \frac12  \delta_0 + \frac12 \Bin(n, 1/2).$$
By a similar argument to that appearing in Section 4.2 of~\cite{BHA22}, independence of the three constituents of $T(\triangle, G_n)$ and the above characterization of $T(\triangle, P_n)$ (which is in the limit a $2$-point discrete distribution) yields that $T(\triangle, G_n)$ is not asymptotically Gaussian.

	\bibliographystyle{plain}
	\bibliography{bib.bib}
	\appendix
	\section{Bounding the Kolmogorov Distance by the \texorpdfstring{$d_2$} Distance}
	\begin{lemma} \label{lem:koltod2}
		Let $\mcZ \sim \mcN(0,1)$ and $Y$ be any random variable, then,
		$$\sup\limits_{t\in \RR}\left|\PP\left(Y \leq t\right) - \PP\left(\mcZ \leq t\right)\right| \lesssim d_2(Y,\mcZ)^{\frac{1}{3}}.$$
	\end{lemma}
	\begin{proof}
		The proof is rather standard and follows, as in \cite[Theorem 3.3]{chen2011normal}, by regularizing the step function. Thus, for a positive small $\delta > 0$, and $t\in \RR$, set
		$$h_{t,\delta}(x) =  \begin{cases} 
			1 & x\leq t \\
			\frac{((t + \delta - x)^3 (6 t^2 + \delta^2 + 3 \delta x + 6 x^2 - 3 t (\delta + 4 x)))}{\delta^5} & t< x < t + \delta \\
			0 & x\geq t + \delta 
		\end{cases}.$$
		The values of $h_{t,\delta}$ in $[t,t+\delta]$ were chosen as the unique minimal degree interpolating polynomial to ensure that $h_{t,\delta}$ is twice continuously differentiable in $\RR$. In particular, a straightforward, but tedious, calculation shows
		$$\|h_{t,\delta}'\|_{\infty} \leq \frac{15}{8\delta} \text{ and } \|h_{t,\delta}''\|_{\infty} \leq \frac{10}{\sqrt{3}\delta^2}.$$
		Thus, for small $\delta$, $\|h_{t,\delta}\|_{C_2} \leq \frac{10}{\sqrt{3}\delta^2}$, and 
		$$\left|\EE\left[h_{t,\delta}(Y)\right] - \EE\left[h_{t,\delta}(\mcZ)\right]\right| \leq \frac{10}{\sqrt{3}\delta^2}d_2(Y,\mcZ).$$
		Since $0<h_{t,\delta}(x) < 1$ for $x \in (t,t+\delta)$ we may conclude,
		\begin{align*}
			\PP\left(Y \leq t\right) -\PP\left(\mcZ \leq t\right) &\leq \EE\left[h_{t,\delta}(Y)\right]-\PP\left(\mcZ \leq t\right) \\
			&= |\EE\left[h_{t,\delta}(Y)\right] -\EE\left[h_{t,\delta}(\mcZ)\right]| + |\EE\left[h_{t,\delta}(\mcZ)\right] -\PP\left(\mcZ \leq t\right)|\\
			&\leq \frac{10}{\sqrt{3}\delta^2}d_2(Y,\mcZ) + \PP\left(t < \mcZ < t + \delta\right)\\
			&\leq \frac{10}{\sqrt{3}\delta^2}d_2(Y,\mcZ) + \frac{\delta}{\sqrt{2\pi}}.
		\end{align*}
		Optimizing over $\delta$, we get $ \PP\left(Y \leq t\right) -\PP\left(\mcZ \leq t\right) \leq 2d_2(X,\mcZ)^{\frac{1}{3}},$ and a similar argument bounds $\PP\left(\mcZ \leq t\right) -\PP\left(Y \leq t\right)$. 
	\end{proof}
\end{document}